%% file: nuclear.tex
\documentclass[a4paper, intlimits, reqno]{amsart}

\usepackage[english]{babel}
\usepackage[T1]{fontenc}
\usepackage[utf8]{inputenc}

\usepackage{amsmath}
\usepackage{amssymb}
\usepackage{MnSymbol}
\usepackage{amsthm}
\allowdisplaybreaks
\usepackage{amsfonts}
\usepackage{mathtools}
\usepackage{enumitem}
\usepackage{url}
\usepackage{dsfont}
\usepackage[numbers]{natbib}
\usepackage{caption}
\usepackage{tikz}
\usetikzlibrary{patterns,calc,arrows}
\tikzstyle{arrow}=[draw, -latex]
\usepackage{prettyref}

\newrefformat{def}{Definition \ref{#1}}
\newrefformat{rem}{Remark \ref{#1}}
\newrefformat{sect}{Section \ref{#1}}
\newrefformat{sub}{Section \ref{#1}}
\newrefformat{prop}{Proposition \ref{#1}}
\newrefformat{thm}{Theorem \ref{#1}}
\newrefformat{chap}{Chapter \ref{#1}}
\newrefformat{cor}{Corollary \ref{#1}}
\newrefformat{par}{Paragraph \ref{#1}}
\newrefformat{ex}{Example \ref{#1}}
\newrefformat{part}{Part \ref{#1}}
\newrefformat{fig}{Figure \ref{#1}}
\newrefformat{app}{Appendix \ref{#1}}
\newrefformat{que}{Question \ref{#1}}
\newrefformat{cond}{Condition \ref{#1}}

\swapnumbers
\newtheoremstyle{dotless}{}{}{\itshape}{}{\bfseries}{}{}{}
\theoremstyle{dotless}
\theoremstyle{plain}
\newtheorem{thm}{Theorem}[section]
\newtheorem{lem}[thm]{Lemma}
\newtheorem{prop}[thm]{Proposition}

\theoremstyle{definition}
\newtheorem{defn}[thm]{Definition}
\newtheorem{rem}[thm]{Remark}

\newtheorem{exa}[thm]{Example}

\newtheorem{conv}[thm]{Convention}

\newtheorem{cond}[thm]{Condition}

\newcommand{\N} {\mathbb{N}}

\newcommand{\R} {\mathbb{R}}
\newcommand{\C} {\mathbb{C}}
\newcommand{\K} {\mathbb{K}}

\newcommand{\D} {\mathbb{D}}

\DeclareMathOperator{\re}{Re}
\DeclareMathOperator{\im}{Im}

\providecommand{\differential}{\mathrm{d}}
\renewcommand{\d}{\differential}
\providecommand{\distance}{\mathrm{D}}
\renewcommand{\D}{\distance}

\newcount\cntloop

\newcommand{\normrec}[2]{%
  \ifnum\cntloop<#1
    \advance\cntloop by 1
    \def\next{\normrec{#1}{\left|\!#2\right|\!}}%
  \else
    \def\next{\left|\!#2\right|}%
  \fi
  \next}

\begin{document}

\title[Nuclearity]{On the nuclearity of weighted spaces of smooth functions}
\author[K.~Kruse]{Karsten Kruse}
\address{TU Hamburg \\ Institut f\"ur Mathematik \\
Am Schwarzenberg-Campus~3 \\
Geb\"aude E \\
21073 Hamburg \\
Germany}
\email{karsten.kruse@tuhh.de}



\date{\today}
\begin{abstract}
Nuclearity plays an important role for the Schwartz kernel theorem to hold and in transferring the surjectivity 
of a linear partial differential operator from scalar-valued to vector-valued functions via tensor product theory. 
In this paper we study weighted spaces $\mathcal{EV}(\Omega)$ of smooth functions 
on an open subset $\Omega\subset\R^{d}$ whose topology is given by a family of weights $\mathcal{V}$. 
We derive sufficient conditions on the weights which make $\mathcal{EV}(\Omega)$ a nuclear space. 
\end{abstract}
\maketitle
\section{Introduction}
\input{intro}
\section{Partition of unity}
\input{partition}
\section{Nuclearity}
\input{Weighted_diff3}
\bibliography{biblio}
\bibliographystyle{plain}
\end{document}

%% file: intro.tex
In this paper we study the relation between a family of weight functions on an open set $\Omega\subset\R^{d}$, $d\in\N$, 
and the nuclearity (see \cite{Gro}, \cite{Pietsch1972}) of the space of infinitely continuously partially 
differentiable functions on $\Omega$ with values in $\K=\R$ or $\C$ whose topology is generated by that family of weights. 
The spaces we want to consider look as follows.

\begin{defn}\label{def:weighted_smooth}
Let $\Omega\subset\R^{d}$ be open and $(\Omega_{n})_{n\in\N}$ a family of 
non-empty sets such that $\Omega_{n}\subset\Omega_{n+1}$ and $\Omega=\bigcup_{n\in\N} \Omega_{n}$.
Let $\mathcal{V}:=(\nu_{n})_{n\in\N}$ be a countable family of positive continuous functions 
$\nu_{n}\colon \Omega \to (0,\infty)$ such that $\nu_{n}\leq\nu_{n+1}$ for all $n\in\N$.
We call $\mathcal{V}$ a (directed) family of continuous weights on $\Omega$ and define
\[
\mathcal{EV}(\Omega):=\{ f \in \mathcal{C}^{\infty}(\Omega, \K)\; | \;\forall\; n \in \N,\,m\in\N_{0}:\; |f|_{n,m}<\infty\}
\]
where $\N_{0}:=\N\cup\{0\}$ and
\begin{equation}\label{eq:intro_seminorm}
|f|_{n,m}:=\sup_{\substack{x \in \Omega_{n}\\ \alpha \in \N_{0}^{d}, \, |\alpha| \leq m}}|\partial^{\alpha}f(x)|\nu_{n}(x).
\end{equation}
\end{defn}

Here $\mathcal{C}^{\infty}(\Omega, \K)$ denotes the space of infinitely continuously partially differentiable 
functions on $\Omega$ with values in $\K$ and $\partial^{\alpha}f$ the $\alpha$-th partial derivative of 
$f$ with respect to the multiindex $\alpha$ of order $|\alpha|$. 

Our goal is to derive sufficient conditions on $\mathcal{V}$ (and $(\Omega_{n})_{n\in\N}$) 
such that $\mathcal{EV}(\Omega)$ becomes a nuclear space. 
Nuclearity implies the approximation property and 
corresponding sufficient conditions for $\mathcal{EV}(\Omega)$ having the approximation property are stated 
in \cite[5.2 Theorem, p.\ 253]{kruse2018_2} and \cite[3.4 Remark, p.\ 239]{kruse2018_2}.
For weighted spaces of smooth functions with a different locally convex topology, e.g.\ where 
the supremum in \eqref{eq:intro_seminorm} is taken over all $\alpha \in \N_{0}^{d}$ or all $n\in\N$, 
conditions for nuclearity are known due to Komatsu \cite[Theorem 2.6, p.\ 44]{Kom7}, Petzsche \cite[3.9 Satz, p.\ 158-159]{Petzsche1978}, 
Braun et al.\ \cite[4.9 Proposition, p.\ 223]{Braun1990} and Boiti et al.\ \cite[Corollary 2.3, p.\ 4]{boiti2019_1}, 
\cite[Theorem 3.6, p.\ 7]{boiti2019_1}, \cite[Theorem 3.3, p.\ 16]{boiti2019_2} for spaces of ultradifferentiable functions 
and due to Mityagin \cite[Corollary, p.\ 322]{Mityagin1960}, \cite[Theorem 3, p.\ 323]{Mityagin1960}, 
Wloka \cite[Satz 6-8, p.\ 88-92]{Wloka1967}, 
Yamanaka \cite{Yamanaka1961}, Nakamura \cite[Theorem 2, p.\ 52]{nakamura1968_1} 
and Funakosi \cite[Theorem 2, p.\ 65]{Funakosi1968} for Gelfand-Shilov spaces. 
Subspaces of $\mathcal{EV}(\Omega)$ consisting of holomorphic functions with exponential growth or decay often appear 
in Paley-Wiener type theorems describing the image of Fourier transformation, 
see e.g.\ \cite[Theorem 4.5, p.\ 162-163]{gel1968}, \cite[Theorem 7.3.1, p.\ 181]{H1}, 
\cite[Theorem 8.1.1, p.\ 368-369]{Kan} and \cite[Theorem 3.1.1, p.\ 485]{Kawai}, and are also the basic spaces for the 
theory of Fourier hyperfunctions, see e.g.\ \cite{Ion/Ka}, \cite{Ito/Nag}, \cite{J}, \cite{ich}, \cite{L2} and \cite{L3}.
An important consequence of nuclearity is that the Schwartz kernel theorem holds (see e.g.\ Petzsche \cite[4.2 Folgerung, p.\ 161-162]{Petzsche1978}). 
In addition, an affirmative answer to the question of nuclearity 
of $\mathcal{EV}(\Omega)$ transfers the surjectivity of a linear partial differential operator
$P(\partial)\colon \mathcal{EV}(\Omega)\to \mathcal{EV}(\Omega)$ with smooth coefficients 
to its corresponding vector-valued counterpart on weighted spaces of smooth functions with values 
in certain locally convex spaces $E$, for example, in the case that $\mathcal{EV}(\Omega)$ and $E$ are Fr\'{e}chet spaces 
by \cite[Satz 10.24, p.\ 255]{Kaballo} and \cite[3.21 Example, p.\ 14]{kruse2017} 
and in other cases using the splitting theory of Vogt \cite{V1} or of Bonet and Doma\'nski \cite{Dom1}. 
$\mathcal{EV}(\Omega)$ is a Fr\'{e}chet space if for every compact set $K\subset\Omega$ there is some $n\in\N$ such that 
$K\subset\Omega_{n}$ (see e.g.\ \cite[3.7 Proposition, p.\ 240]{kruse2018_2}).
Applications to the Cauchy-Riemann operator $\overline{\partial}$ are given in \cite[5.24 Theorem, p.\ 95]{ich} 
and \cite[4.9 Corollary, p.\ 21]{kruse2018_5}.

Gelfand and Vilenkin treat the spaces $K\{\nu_{n}\}:=\mathcal{EV}(\Omega)$ with $\Omega_{n}=\R$, 
$\nu_{n}\geq 1$, monotonically increasing $\nu_{n}(|\cdot|)$ and $\nu_{n}\in\mathcal{C}^{\infty}(\R,\R)$ for all $n\in\N$. 
If 
\begin{enumerate}
\item [(N.1)] for any $n\in\N$ and $\alpha\in\N_{0}$ there are $C>0$ and $k\in\N$ such that 
$|\partial^{\alpha}\nu_{n}|\leq C \nu_{k}$, and if 
\item [(N.2)] for any $n\in\N$ there is $k\in\N$ such that $\lim_{|x|\to\infty}\nu_{n}(x)/\nu_{k}(x)=0$ 
and $\nu_{n}/\nu_{k}$ is an element of the Lebesgue space $L^{1}(\R)$, 
\end{enumerate}
then $K\{\nu_{n}\}$ is nuclear by \cite[Ch.\ I, Sect.\ 3.6, Theorem 7, p.\ 82]{Gelfand1964}. 
In particular, the conditions are fulfilled for $\nu_{n}(x)=(1+x^{2})^{n}$, $x\in\R$, 
implying the nuclearity of the classical Schwartz space (in one variable). 
The downside of Gelfand's and Vilenkin's conditions is that 
$\nu_{n}$ has to be smooth and that $\Omega_{n}=\R$ for all $n\in\N$. 

For the subspace $\mathcal{OV}(\Omega)$ of $\mathcal{EV}(\Omega)$ of holomorphic functions on $\Omega\subset\C$ 
sufficient conditions for the nuclearity of $\mathcal{OV}(\Omega)$ in terms of $\mathcal{V}$ 
are derived by Wloka in \cite[Satz 9a, p.\ 178]{Wloka1966} using a technique of reproducing kernels 
(cf.\ \cite[p.\ 722-723]{Wloka1965}). 
Wloka's conditions for $\mathcal{EV}(\Omega)$ in \cite[Satz 3, p.\ 82]{Wloka1967} contain 
for a family $(\Omega_{n})_{n\in\N}$ of open connected sets and $1\leq p<\infty$ the conditions 
\begin{enumerate}
 \item [$(N_{1}^{p})$] $\int_{\Omega_{n}}\bigl(\frac{\nu_{n}(x)}{\nu_{k}(x)}\bigr)^{p}\d x<\infty$ for $n<k$,
 \item [$(N_{2}^{p})$] for $k>n+d/p$ there is $C>0$ such that for all $t\in\Omega_{n}$ there is $r_{t}>0$ 
 such that $B(t,r_{t})\subset \Omega_{k}$ and
 \[
  A(r_{t})\frac{\nu_{n}(t)}{\nu_{k}(x)}\leq C,\quad x\in B(t,r_{t}),
 \]
 where $B(t,r_{t})$ is the ball around $t$ with radius $r_{t}$ and $A(r_{t})$ is the embedding constant from the Sobolev embedding 
 $W^{k}_{p}(B(t,r_{t}))\hookrightarrow \mathcal{C}(\overline{B(t,r_{t})})$, see \cite[\S1.8, Theorem 1, p.\ 54]{Sobolev1991}. 
\end{enumerate}
$\mathcal{C}(\overline{B(t,r_{t})})$ denotes the space of continuous functions on the closure $\overline{B(t,r_{t})}$ and 
$W^{k}_{p}(B(t,r_{t}))$ the Sobolev space of (equivalence classes of) functions on $B(t,r_{t})$ such that 
all weak partial derivatives up to order $k$ are in the Lebesgue space $L^{p}(B(t,r_{t}))$.
The involvement of the Sobolev embedding constants $A(r_{t})$ makes $(N_{2}^{p})$ less applicable 
since only knowing their sheer existence might not be helpful, and even if one explicitly knows them, they might depend on $t$. 
For example, an explicit Sobolev embedding constant for $k=1>d/p$ can be found in \cite[Theorem 2.E, p.\ 200]{Talenti1994} 
but it still depends on $r_{t}$ and thus possibly on $t$. 
Therefore Wloka only applies his conditions (with some additional assumptions) in the case 
where he can take one $r_{t}$ for all $t\in\Omega_{n}$, namely, in the case $\Omega_{n}=\R^{d}$ for all $n\in\N$ 
with $r_{t}=1$ for all $t\in\R^{d}$, see \cite[Satz 4, Folgerung, p.\ 85]{Wloka1967}.  

Triebel considers generalised Schwartz spaces $\mathcal{S}_{\rho}(\Omega):=\mathcal{EV}(\Omega)$ with
open, connected $\Omega\subset\R^{d}$, $\Omega_{n}=\Omega$ and $\nu_{n}(x)=\rho(x)^{n}$, $x\in\Omega$, for all $n\in\N$ 
where $\rho\in\mathcal{C}^{\infty}(\Omega,\R)$, $\rho\geq 1$. 
The spaces $\mathcal{S}_{\rho}(\Omega)$ are generated by partial differential operators and 
Triebel obtains a sufficient condition for nuclearity by using interpolation theory. Namely, if $\rho$ satisfies the conditions
\begin{enumerate}
\item [(T.1)] $\forall\;\alpha\in\mathbb{N}_{0}^{d}\;\exists\;C>0\;\forall\;x\in\Omega:\;
|\partial^{\alpha}\rho(x)|\leq C\rho(x)^{1+|\alpha|}$,
\item [(T.2)] $\forall\;C>0\;\exists\;\varepsilon>0,\,r>0\;\forall\;x\in\Omega:\; 
\rho(x)>C \quad\text{if}\;\d_{\partial\Omega}(x)\leq \varepsilon\;\text{or if}\;|x|\geq r$,
\item [(T.3)] $\exists\;a\geq 0:\;\rho^{-a}\in L^{1}(\Omega)$,
\end{enumerate}
where $\d_{\partial\Omega}(x)$ is the distance of $x\in\Omega$ to the boundary $\partial\Omega$ of $\Omega$, 
then $\mathcal{S}_{\rho}(\Omega)$ is nuclear by \cite[8.3.2 Theorem 1, p.\ 481]{triebel1978}, 
which generalises \cite[Satz 4, p.\ 171]{Triebel1967} ($\Omega$ bounded) 
and \cite[Satz 5, p.\ 302]{Triebel1970} ($\rho(x)\geq C|x|^{a}$ for all $x\in\Omega$ with some $C>0$ and $a>0$).
These conditions are fulfilled for $\Omega=\R^{d}$ and $\rho(x)=1+|x|^{2}$, $x\in\R^{d}$, yielding
the nuclearity of the classical Schwartz space as well. 
Similar to Gelfand and Vilenkin, the drawback of Triebel's conditions is that 
$\rho$ has to be smooth and all $\Omega_{n}$ have to coincide. 

Let us outline the content of the paper. In Section 2 we introduce our sufficient conditions on $\mathcal{V}$ and 
$(\Omega_{n})_{n\in\N}$, stated in \prettyref{cond:weights} and \prettyref{rem:cond_sets}, 
guaranteeing the nuclearity of $\mathcal{EV}(\Omega)$ and explore many examples. 
These conditions are more likely to be applicable as they overcome the disadvantages of 
Gelfand's and Vilenkin's, Wloka's and Triebel's conditions.
In the same section we construct a partition of unity in \prettyref{lem:partition_unity} 
which is then used in our main \prettyref{thm:nuclear}
in Section 3 to prove the sufficiency for nuclearity of our conditions.


%% file: partition.tex
We begin with our sufficient conditions on the weight functions which are modifications of the conditions 
$(1.1)$-$(1.3)$ in \cite[p.\ 204]{L4} where the case $\Omega_{n}=\Omega$ for all $n\in\N$ is considered.
We set $\|x\|_{\infty}:=\max_{1\leq i\leq d}|x_{i}|$ for $x=(x_{i})\in\R^{d}$ and use the 
convention $\inf\varnothing:=\infty$.

\begin{cond}\label{cond:weights}
Let $\mathcal{V}:=(\nu_{n})_{n\in\N}$ be a family of continuous weights on an open set $\Omega\subset\R^{d}$ 
and $(\Omega_{n})_{n\in\N}$ a family of non-empty Lebesgue measurable sets such that $\Omega_{n}\subset\Omega_{n+1}$
and $\Omega=\bigcup_{n\in\N} \Omega_{n}$. For every $n\in\N$ set 
\[
\d^{\infty}_{n+1}\colon \Omega_{n}\to [0,\infty],\;
\d^{\infty}_{n+1}(x):=\operatorname{inf}\{\|x-\zeta\|_{\infty}\;|\;\zeta\in\partial\Omega_{n+1}\}.
\]
Let us assume that $(\Omega_{n})_{n\in\N}$ is such that for every $k\in\N$ there exists $r_{k}\colon\Omega_{k}\to (0,1]$ with 
$0<r_{k}(x)<\d^{\infty}_{k+1}(x)$ for all $x\in\Omega_{k}$ and for any $n\in\N$ there exists $\psi_{n}\in L^{1}(\Omega_{k})$ with $\psi_{n}>0$ 
and there exist $\N\ni I_{j}(n)\geq n$ and $A_{j}(n)>0$ for $j=1,2,3$ such that for any $x\in\Omega_{k}$:
\begin{enumerate}
  \item [$(\omega.1)$] $\sup_{\zeta\in\R^{d},\,\|\zeta\|_{\infty}\leq r_{k}(x)}\nu_{n}(x+\zeta)
  \leq A_{1}(n)\inf_{\zeta\in\R^{d},\,\|\zeta\|_{\infty}\leq r_{k}(x)}\nu_{I_{1}(n)}(x+\zeta)$,
  \item [$(\omega.2)$] $\nu_{n}(x)\leq A_{2}(n)\psi_{n}(x)\nu_{I_{2}(n)}(x)$,
  \item [$(\omega.3)$] $\nu_{n}(x)\leq A_{3}(n)r_{k}(x)\nu_{I_{3}(n)}(x)$.
\end{enumerate}
\end{cond}

In the case $\Omega_{n}=\Omega$ for all $n\in\N$ conditions like $(\omega.1)$-$(\omega.3)$ 
appear in \cite[2.1 Definition, p.\ 67]{Meise1985} and \cite[1.1 Definition, p.\ 343]{MeiseTaylor1987} 
for $\Omega=\C^{d}$ and $r:=r_{k}=1$ and in \cite{hoermander1967} for $\Omega\subset\C^{d}$ 
where the assumption (ii) of \cite[Theorem 1, p.\ 943]{hoermander1967} means that one may take 
$r(z):=r_{k}(z)=e^{-\mu_{N}(z)-C}$, $z\in\Omega$, for some $N\in\N$ and $C>0$ if the family $\mathcal{V}$ 
is given by $\nu_{n}(z):=e^{-\mu_{n}(z)}$ with a sequence of positive continuous plurisubharmonic functions 
$(\mu_{n})$. We use the following convention from \cite[1.1 Convention, p.\ 205]{L4}.

\begin{conv}
We often delete the number $n$ counting the seminorms (e.g.\ $I_{j}=I_{j}(n)$ or $A_{j}=A_{j}(n)$) 
and indicate compositions with the functions $I_{j}$ only in the index 
(e.g.\ $I_{23}=I_{2}(I_{3}(n))$).
\end{conv}

The conditions $(\omega.1)$-$(\omega.3)$ are closed under multiplication, more generally, we have:

\begin{rem}\label{rem:closed_multipl}
\begin{enumerate}
\item [a)]Let $\widehat{\mathcal{V}}:=(\widehat{\nu}_{n})_{n\in\N}$ be a family fulfilling 
$(\omega.1)$-$(\omega.3)$ and $\widetilde{\mathcal{V}}:=(\widetilde{\nu}_{n})_{n\in\N}$ a family fulfilling 
$(\omega.1)$ and $(\omega.3)$. 
Then $\mathcal{V}:=(\widehat{\nu}_{n}\widetilde{\nu}_{n})_{n\in\N}$ fulfils $(\omega.1)$-$(\omega.3)$.
\item [b)] If $\inf_{x\in\Omega_{k}}r_{k}(x)>0$ for every $k\in\N$, then $(\omega.3)$ is fulfilled.
\end{enumerate}
\end{rem}
\begin{proof}
\begin{enumerate}
 \item [a)] Let $k,n\in\N$, $x\in\Omega_{k}$ and define $r_{k}:=\min(\widehat{r}_{k},\widetilde{r}_{k})$.
\begin{enumerate}
  \item [$(\omega.1)$] We set
   $I_{1}(n):=\max(\widehat{I}_{1}(n),\widetilde{I}_{1}(n))$ and obtain
  \begin{flalign*}
   &\quad\;\sup\{(\widehat{\nu}_{n}\widetilde{\nu}_{n})(x+\zeta)\;|\;\|\zeta\|_{\infty}\leq r_{k}(x)\}\\
   &\leq \widetilde{A}_{1}\widehat{A}_{1}\inf\{\widetilde{\nu}_{\widetilde{I}_{1}}(x+\zeta)\;|\;
   \|\zeta\|_{\infty}\leq \widetilde{r}_{k}(x)\}
   \inf\{\widehat{\nu}_{\widehat{I}_{1}}(x+\eta)\;|\;\|\eta\|_{\infty}\leq \widehat{r}_{k}(x)\}\\
   &\leq\widetilde{A}_{1}\widehat{A}_{1}\inf\{\widetilde{\nu}_{\widetilde{I}_{1}}(x+\zeta)\widehat{\nu}_{\widehat{I}_{1}}(x+\eta)
   \;|\;\|\zeta\|_{\infty},\,\|\eta\|_{\infty}\leq r_{k}(x)\}\\
   &\leq\widetilde{A}_{1}\widehat{A}_{1}\inf\{\widetilde{\nu}_{\widetilde{I}_{1}}(x+\zeta)\widehat{\nu}_{\widehat{I}_{1}}(x+\zeta)\;|\;
   \|\zeta\|_{\infty}\leq r_{k}(x)\}\\
   &\leq\widetilde{A}_{1}\widehat{A}_{1}\inf\{(\widetilde{\nu}_{I_{1}}\widehat{\nu}_{I_{1}})(x+\zeta)\;|\;
   \|\zeta\|_{\infty}\leq r_{k}(x)\}.
  \end{flalign*}
  \item [$(\omega.2)$] We have
  \[
   (\widehat{\nu}_{n}\widetilde{\nu}_{n})(x)
   \leq\widehat{A}_{2}\widehat{\psi}_{n}(x)\widehat{\nu}_{\widehat{I}_{2}}(x)\widetilde{\nu}_{n}(x)
   \leq\widehat{A}_{2}\widehat{\psi}_{n}(x)(\widehat{\nu}_{\widehat{I}_{2}}\widetilde{\nu}_{\widehat{I}_{2}})(x).
  \]
  \item [$(\omega.3)$] We set $I_{3}(n):=\max(\widehat{I}_{3}(n),\widetilde{I}_{3}(n))$ and get
  \[
  (\widehat{\nu}_{n}\widetilde{\nu}_{n})(x)
  \leq \widehat{A}_{3}\widetilde{A}_{3}\widehat{r}_{k}(x)\widetilde{r}_{k}(x)
  \widehat{\nu}_{\widehat{I}_{3}}(x)\widetilde{\nu}_{\widetilde{I}_{3}}(x)
  \leq \widehat{A}_{3}\widetilde{A}_{3}r_{k}(x)(\widehat{\nu}_{I_{3}}\widetilde{\nu}_{I_{3}})(x)
  \]
  since $\widehat{r}_{k}\widetilde{r}_{k}\leq r_{k}$ as $\widehat{r}_{k},\,\widetilde{r_{k}}\leq 1$.
\end{enumerate}
\item [b)] Let $c_{k}:=\inf_{x\in\Omega_{k}}r_{k}(x)>0$. Then we get
\[
 \nu_{n}(x)=r_{k}(x)^{-1}r_{k}(x)\nu_{n}(x)\leq c_{k}^{-1}r_{k}(x)\nu_{n}(x).
\]
\end{enumerate}
\end{proof}

\begin{rem}\label{rem:r_n_loc_bound_away}
Let $\mathcal{V}$ fulfil $(\omega.3)$. The functions $r_{n}$, $n\in\N$, are locally bounded away from zero on $\Omega_{n}$, 
i.e.\ for every compact set $K\subset \Omega_{n}$ there is $\varepsilon>0$ such that $r_{n}(x)\geq \varepsilon$ for all $x\in K$. 
Indeed, if $K\subset\Omega_{n}$ is compact (w.r.t.\ the topology induced by $\Omega$), then $K$ is compact in $\Omega$. 
It follows that 
\[
r_{n}(x)\underset{(\omega.3)}{\geq} \frac{1}{A_{3}(n)}\frac{\nu_{n}(x)}{\nu_{I_{3}(n)}(x)}
\geq \frac{1}{A_{3}(n)}\frac{\inf_{z\in K}\nu_{n}(z)}{\sup_{z\in K}\nu_{I_{3}(n)}(z)}>0,\quad x\in K,
\] 
since the members of the family $\mathcal{V}$ are locally bounded and 
locally bounded away from zero on $\Omega$.
\end{rem}

The following functions generated from $r_{n}$ play a big role in the construction of our partition of unity 
which is used to derive the nuclearity of $\mathcal{EV}(\Omega)$ if the family $\mathcal{V}$ of continuous weights 
fulfils $(\omega.1)$-$(\omega.3)$.

\begin{defn}
Let $\mathcal{V}$ fulfil $(\omega.3)$. For $n\in\N$ set $r_{n,0}:=r_{n}$ and for $k\in\N$ let $r_{n,k}$ be given by 
\[
 r_{n,k}(z):=\inf\{r_{n,k-1}(\eta)\;|\;\eta\in\Omega_{n}:\;\|\eta-z\|_{\infty}\leq r_{n}(\eta)\;\text{or}\;
 \|\eta-z\|_{\infty}\leq r_{n}(z)\},
 \quad z\in\Omega_{n}.
\]
\end{defn}

For our partition of unity we need the positivity of $r_{n,k}$ for all $n\in\N$ and $k\in\N_{0}$, which is guaranteed 
by \prettyref{rem:r_n_loc_bound_away} and the proposition below. 

\begin{prop}\label{prop:loc_bound_away}
Let $U\subset\R^{d}$ be non-empty, $r\colon U\to [0,\infty)$ bounded and $f\colon U\to [0,\infty)$ 
locally bounded away from zero on $U$.
If
\begin{enumerate}
 \item [(i)] $f$ is bounded away from zero on $U$, i.e.\ $\inf_{x\in U}f(x)>0$, or if
 \item [(ii)] $U$ is closed, or if
 \item [(iii)] $U$ is open and $r$ continuous such that $0<r(x)<\d^{\infty}_{\partial U}(x)$ for all $x\in U$ where 
 \[
  \d^{\infty}_{\partial U}\colon U\to (0,\infty],\; \d^{\infty}_{\partial U}(x):=\inf\{\|x-\zeta\|_{\infty}\;|\;\zeta\in \partial U\},
 \]
\end{enumerate}
then 
\[
g\colon U\to [0,\infty),\; g(x):=\inf\{f(\eta)\;|\;\eta\in U:\;\|\eta-x\|_{\infty}\leq r(\eta)\;
 \text{or}\;\|\eta-x\|_{\infty}\leq r(x)\},
\]
is bounded away from zero on $U$ in (i) and locally bounded away from zero on $U$ in (ii) and (iii). In particular, $g>0$ on $U$.
\end{prop}
\begin{proof}
 In case (i), this is obviously true because $\inf_{x\in U}g(x)\geq \inf_{x\in U}f(x)>0$. 
 Let us turn to case (ii) and (iii) and assume the contrary. Then there is a compact set $K\subset U$ such that 
 for every $n\in\N$ there is $x_{n}\in K$ with $g(x_{n})\leq 1/(2n)$. It follows that for every $n\in\N$ 
 there is $\eta_{n}\in\{\eta\in U:\;\|\eta-x_{n}\|_{\infty}\leq r(\eta)\;
 \text{or}\;\|\eta-x_{n}\|_{\infty}\leq r(x_{n})\}$ such that $f(\eta_{n})\leq g(x_{n})+1/(2n)\leq 1/n$. 
 The sequence $(\eta_{n})_{n\in\N}$ is bounded since $r$ and $(x_{n})_{n\in\N}$ are bounded. 
 Hence there are a subsequence $(\eta_{n_{k}})_{k\in\N}$ and $\eta\in \overline{U}$ such that 
 $\eta_{n_{k}}\to\eta$ by the Bolzano-Weierstra{\ss} theorem. 
 If $\eta\in U$, then the set $K_{0}:=\{\eta_{n_{k}}\;|\;k\in\N\}\cup\{\eta\}$ 
 is compact in $U$ and $f(\eta_{n_{k}})\leq 1/n_{k}$ for every $k\in\N$, which contradicts $f$ being locally 
 bounded away from zero on $U$. If $U$ is closed, then $U=\overline{U}$, which settles the case (ii). 
 Let $\eta\in\partial U$ and (iii) hold. Then $\eta_{n_{k}}\to\eta$ implies
 \[
  0<r(\eta_{n_{k}})<\d^{\infty}_{\partial U}(\eta_{n_{k}})\to 0,
 \]
 yielding $r(\eta_{n_{k}})\to 0$. 
 If there are infinitely many $k\in\N$ such that $\|\eta_{n_{k}}-x_{n_{k}}\|_{\infty}\leq r(\eta_{n_{k}})$, then 
 there is a subsequence $(x_{n_{k_{p}}})_{p\in\N}$ in the compact set $K\subset U$ which converges 
 to $\eta\in\partial U$. Since $U$ is open, this is a contradiction. 
 If there are infinitely many $k\in\N$ such that $\|\eta_{n_{k}}-x_{n_{k}}\|_{\infty}\leq r(x_{n_{k}})$, 
 there is a subsequence $(x_{n_{k_{p}}})_{p\in\N}$ with $\|\eta_{n_{k_{p}}}-x_{n_{k_{p}}}\|_{\infty}\leq r(x_{n_{k_{p}}})$ 
 for all $p\in\N$. This bounded subsequence has again a subsequence $(x_{n_{k_{p_{q}}}})_{q\in\N}$ which converges 
 to some $x\in K$. Since $r$ is continuous on $U$, we have $r(x_{n_{k_{p_{q}}}})\to r(x)$. From 
 \begin{align*}
  \|\eta_{n_{k_{p_{q}}}}-x\|_{\infty}&\leq \|\eta_{n_{k_{p_{q}}}}-x_{n_{k_{p_{q}}}}\|_{\infty}+\|x_{n_{k_{p_{q}}}}-x\|_{\infty}
  \leq r(x_{n_{k_{p_{q}}}})+\|x_{n_{k_{p_{q}}}}-x\|_{\infty}
  \to r(x)
 \end{align*}
 follows $\|\eta-x\|_{\infty}\leq r(x)<\d^{\infty}_{\partial U}(x)$. However, this means that $\eta\in U$, which is a 
 contradiction.
\end{proof}

\begin{rem}\label{rem:cond_sets}
Let $\mathcal{V}$ fulfil $(\omega.3)$. From \prettyref{rem:r_n_loc_bound_away} and \prettyref{prop:loc_bound_away}
follows by induction that $r_{n,k}>0$ on $\Omega_{n}$ for every $n\in\N$ and $k\in\N_{0}$ if 
\begin{enumerate}
 \item [(s.1)] $r_{n}$ is bounded away from zero on $\Omega_{n}$ for all $n\in\N$, 
 in particular, if $r_{n}$ is constant for all $n\in\N$, or if
 \item [(s.2)] $\Omega_{n}$ is closed in $\R^{d}$ for all $n\in\N$, or if 
 \item [(s.3)] $\Omega_{n}=\Omega$ and $r_{n}$ is continuous for all $n\in\N$.
\end{enumerate}
\end{rem}

\begin{exa}\label{ex:families_of_weights_1} 
Let $\Omega\subset\R^{d}$ be open and $(\Omega_{n})_{n\in\N}$ a family of non-empty Lebesgue measurable 
sets such that $\Omega_{n}\subset\Omega_{n+1}$
and $\Omega=\bigcup_{n\in\N} \Omega_{n}$. For $n\in\N$ set 
\[
 \D^{\infty}_{n+1}:=\inf\{\|x-\zeta\|_{\infty}\;|\;x\in\Omega_{n},\,\zeta\in\partial\Omega_{n+1}\}.
\]
Let $\Omega_{n}=\R^{d}$ for all $n\in\N$ or $0<\D^{\infty}_{n+1}<\infty$ for all $n\in\N$, 
$(a_{n})_{n\in\N}$ be strictly increasing such that $a_{n}\geq 0$ for all $n\in\N$ or 
$a_{n}\leq 0$ for all $n\in\N$. 
The family $\mathcal{V}:=(\nu_{n})_{n\in\N}$ of positive continuous functions on $\Omega$ given by 
\[
 \nu_{n}\colon\Omega\to (0,\infty),\;\nu_{n}(x):=e^{a_{n}\mu(x)},
\]
with some function $\mu\colon\Omega\to[0,\infty)$ fulfils $\nu_{n}\leq\nu_{n+1}$ for all $n\in\N$ and
  \begin{enumerate}
  \item [(i)] $(\omega.1)$ and $(\omega.3)$ if there is $\delta>0$ such that $|\mu(x)-\mu(y)|\leq 1$
	for all $x,y\in\Omega$ with $\|x-y\|_{\infty}\leq \delta$, in particular, if $\mu$ is uniformly continuous.
  \item [(ii)] $(\omega.1)$-$(\omega.3)$ if $\mu$ is like in (i) and if there are a real sequence 
  $(c_{n})_{n\in\N}$ and $0<\gamma\leq 1$ such that $|x|^{\gamma}\leq\mu(x)+c_{n}$ for all $x\in\Omega_{n}$.
  \item [(iii)] $(\omega.1)$-$(\omega.3)$ if there is some $m\in\N$ such that $\mu(x)=|x|^{m}$ for all $x\in\Omega$.
  \item [(iv)] $(\omega.1)$-$(\omega.3)$ if $\Omega_{n}=\R^{d}$, $a_{n}=n/2$ for all $n\in\N$ and 
  $\mu(x)=\ln(1+|x|^{2})$, $x\in\R^{d}$.
  \item [(v)] $(\omega.1)$-$(\omega.3)$ if $\Omega_{n}$ is bounded for all $n\in\N$ and $\mu=0$.
  \end{enumerate}
In addition, $r_{k}$ can be chosen to be 
  \begin{enumerate}
  \item [(1)] continuous for all $k\in\N$ in (i)-(v), 
  \item [(2)] constant for all $k\in\N$ in (i), (ii), (iv) and (v),
  \item [(3)] constant for all $k\in\N$ in (iii) if either
  $\lim_{n\to\infty}a_{n}=\infty$, $a_{n}\geq 0$ for all $n\in\N$ or 
  $\lim_{n\to\infty}a_{n}=0$, $a_{n}\leq 0$ for all $n\in\N$.
  \end{enumerate}
\end{exa}
\begin{proof}
\begin{enumerate}
\item [(i)] Let $k\in\N$. There is $\delta>0$ such that $|\mu(x)-\mu(y)|\leq 1$
	for all $x,y\in\Omega$ with $\|x-y\|_{\infty}\leq \delta$. Now, we set $r_{k}(x):=\min(\delta,1)$ 
	for every $x\in\Omega_{k}$ if $\Omega_{n}=\R^{d}$ for all $n\in\N$, and 
	$r_{k}(x):=\min(\delta, \D^{\infty}_{k+1},1)/2$ for every $x\in\Omega_{k}$ if $0<\D^{\infty}_{n+1}<\infty$ for all $n\in\N$. 
	Let $x\in\Omega_{k}$ and $\|\zeta\|_{\infty},\|\eta\|_{\infty}\leq r_{k}(x)$. We have 
	\[
	|\mu(x+\zeta)-\mu(x+\eta)|\leq |\mu(x+\zeta)-\mu(x)|+|\mu(x)-\mu(x+\eta)|\leq 2.
	\] 
  If $a_{n}\geq 0$ for all $n\in\N$, it follows that
	\[
	a_{n}\mu(x+\zeta)\leq 2a_{n}+a_{n}\mu(x+\eta),
	\]
	and if $a_{n}\leq 0$ for all $n\in\N$, then 
	\[
	a_{n}\mu(x+\zeta)=-|a_{n}|\mu(x+\zeta)\leq 2|a_{n}|+a_{n}\mu(x+\eta).
	\]
	Hence we obtain
	\[
	e^{a_{n}\mu(x+\zeta)}\leq e^{2|a_{n}|}e^{a_{n}\mu(x+\eta)},
	\]
	implying that $(\omega.1)$ holds. Further, $(\omega.3)$ is fulfilled by virtue of \prettyref{rem:closed_multipl} b).
	\item [(ii)] Due to (i), we only need to show that $\mathcal{V}$ satisfies $(\omega.2)$. For $n,p\in\N$ we define 
	\[
	 \psi_{n,p}\colon\R^{d}\to [0,\infty),\; \psi_{n,p}(x):=\frac{1}{(1+|x|^2)^{p}},
	\]
  and remark that $\psi_{n,p}\in L^{1}(\R^{d})$ if $d\leq 2p-1$. There is $C_{0}>0$ such that for all $x\in\Omega_{k}$ 
	\begin{align*}
	\psi_{n,d}(x)^{-1}&=(1+|x|^{d})^{2}\leq C_{0}e^{(a_{n+1}-a_{n})|x|^{\gamma}}\leq C_{0}e^{(a_{n+1}-a_{n})(c_{k}+\mu(x))}\\
	&=C_{1}e^{(a_{n+1}-a_{n})\mu(x)}
	\end{align*}
	with $C_{1}:=C_{0}e^{(a_{n+1}-a_{n})c_{k}}$, implying
	\[
	\nu_{n}(x)=e^{a_{n}\mu(x)}\leq C_{1}\psi_{n,d}(x)e^{a_{n+1}\mu(x)}=C_{1}\psi_{n,d}(x)\nu_{n+1}(x)
	\]
	for all $x\in\Omega_{k}$ and thus $(\omega.2)$.
	\item [(iii.1)] We start with the case of no further restrictions on the sequence $(a_{n})_{n\in\N}$. 
	Let $k\in\N$ and choose $p\in\N$ such that $d\leq 2p-1$ and $m\leq 2p+1$.
	We set $r_{k}(x):=(1+|x|^{2})^{-p}$ for every $x\in\Omega_{k}$ 
	if $\Omega_{n}=\R^{d}$ for all $n\in\N$, and $r_{k}(x):=\min((1+|x|^{2})^{-p},\D^{\infty}_{k+1})/2$ for every 
	$x\in\Omega_{k}$ if $0<\D^{\infty}_{n+1}<\infty$ for all $n\in\N$. 
	Let $x\in\Omega_{k}$. For $\|\zeta\|_{\infty},\|\eta\|_{\infty}\leq r_{k}(x)$ we obtain by the binomial theorem
	\begin{align*}
	 \mu(x+\zeta)&\leq\left(|x+\eta|+|\zeta-\eta|\right)^{m}=\sum_{j=0}^{m}\dbinom{m}{j}|x+\eta|^{m-j}|\zeta-\eta|^{j}\\
	 &\leq |x+\eta|^{m}+\sum_{j=1}^{m}\dbinom{m}{j}(|x|+d^{1/2}r_{k}(x))^{m-j}2^{j}d^{j/2}r_{k}(x)^{j}\\
	 &= \mu(x+\eta)+\sum_{j=1}^{m}\dbinom{m}{j}2^{j}d^{j/2}\sum_{l=0}^{m-j}\dbinom{m-j}{l}d^{l/2}|x|^{m-j-l}r_{k}(x)^{j+l}.
	\end{align*}
	Further, we have $|x|^{m-j-l}r_{k}(x)^{j+l}<1$ if $|x|<1$, and 
	\[
	 |x|^{m-j-l}r_{k}(x)^{j+l}\leq \frac{|x|^{m-j-l}}{(1+|x|^{2})^{p(j+l)}}\leq\frac{|x|^{m-1}}{(1+|x|^{2})^{p}}\leq 1
	\]
  if $|x|\geq 1$ and $1\leq j+l\leq m$, since $m-1\leq 2p$. Hence there is $C_{1}>0$ such that 
  $\mu(x+\zeta)\leq \mu(x+\eta)+C_{1}$, resulting in $(\omega.1)$. 
  Moreover, there is $C_{2}>0$ such that 
  \[
   r_{k}(x)^{-1}e^{(a_{n}-a_{n+1})|x|^{m}}= (1+|x|^{2})^{p}e^{(a_{n}-a_{n+1})|x|^{m}}\leq C_{2}
  \]
  resp.\
  \[
   r_{k}(x)^{-1}e^{(a_{n}-a_{n+1})|x|^{m}}\leq 2(1/\D^{\infty}_{k+1}+(1+|x|^{2})^{p})e^{(a_{n}-a_{n+1})|x|^{m}}\leq C_{2}
  \]
  for all $x\in\Omega_{k}$, implying $(\omega.3)$. The choice of $\psi_{n,p}$ from (ii) yields $(\omega.2)$ 
  since $d\leq 2p-1$, $r_{k}(x)\leq\psi_{n,p}(x)$ for all $x\in\Omega_{k}$ and $(\omega.3)$ is satisfied.
  \item [(iii.2)] Now, we consider case (3), i.e.\ $\lim_{n\to\infty}a_{n}=\infty$, $a_{n}\geq 0$ for all $n\in\N$ or 
  $\lim_{n\to\infty}a_{n}=0$, $a_{n}\leq 0$ for all $n\in\N$.
	Let $k\in\N$. We set $r_{k}(x):=1$ for every $x\in\Omega_{k}$ 
	if $\Omega_{n}=\R^{d}$ for all $n\in\N$, and $r_{k}(x):=\min(1,\D^{\infty}_{k+1})/2$ for every 
	$x\in\Omega_{k}$ if $0<\D^{\infty}_{n+1}<\infty$ for all $n\in\N$. 
	Let $x\in\Omega_{k}$. For $\|\zeta\|_{\infty},\|\eta\|_{\infty}\leq r_{k}(x)$ we obtain as above
	\begin{align*}
	 \mu(x+\zeta)&\leq\left(|x+\eta|+|\zeta-\eta|\right)^{m}=\sum_{j=0}^{m}\dbinom{m}{j}|x+\eta|^{m-j}|\zeta-\eta|^{j}\\
	 &\leq \sum_{j=0}^{m}\dbinom{m}{j}|x+\eta|^{m-j}2^{j}d^{j/2}
	  \leq \sum_{j=0}^{m}\dbinom{m}{j}(1+|x+\eta|^{m})2^{j}d^{j}\\
	 &= (1+2d)^{m}(1+\mu(x+\eta)).
	\end{align*}
	If $\lim_{n\to\infty}a_{n}=\infty$ and $a_{n}\geq 0$ for all $n\in\N$ resp.\ 
	$\lim_{n\to\infty}a_{n}=0$ and $a_{n}\leq 0$ for all $n\in\N$, then for every $n\in\N$ 
	there is $J_{1}(n)\in\N$ such that $a_{n}(1+2d)^{m}\leq a_{J_{1}(n)}$ 
	resp.\ $a_{n}/(1+2d)^{m}\leq a_{J_{1}(n)}$. In the first case, it follows that 
	\[
	a_{n}\mu(x+\zeta)\leq a_{n}(1+2d)^{m}(1+\mu(x+\eta))\leq a_{J_{1}(n)}\mu(x+\eta)+a_{J_{1}(n)},
	\]
	and in the second, that
	\[
	a_{n}\mu(x+\eta)\leq \frac{a_{n}}{(1+2d)^{m}}\mu(x+\zeta)-a_{n}\leq a_{J_{1}(n)}\mu(x+\zeta)-a_{n},
	\]
	which yields $(\omega.1)$. Moreover, we choose $p\in\N$ such that $d\leq 2p-1$ and $\psi_{n,p}$ from (ii). 
	Then there is $C>0$ such that 
  \[
   \psi_{n,p}(x)^{-1}e^{(a_{n}-a_{n+1})|x|^{m}}= (1+|x|^{2})^{p}e^{(a_{n}-a_{n+1})|x|^{m}}\leq C
  \]
  for all $x\in\Omega_{k}$, implying $(\omega.2)$. 
  Further, $(\omega.3)$ is valid due to \prettyref{rem:closed_multipl} b). 
	\item [(iv)] Let $k\in\N$. We choose $\psi_{n,d}$ like in (ii) and set $r_{k}(x):=1$, $x\in\R^{d}$. 
	Then we have for $x\in\R^{d}$ and $\|\zeta\|_{\infty},\|\eta\|_{\infty}\leq r_{k}(x)$
	\begin{align*}
	 \nu_{n}(x+\zeta)&\leq (1+(|x+\eta|+|\zeta-\eta|)^{2})^{n/2}\\
	 &\leq (1+|x+\eta|^{2}+4d^{1/2}|x+\eta|+4d)^{n/2}\\
	 &\leq (1+|x+\eta|^{2}+4d^{1/2}(1+|x+\eta|^{2})+4d)^{n/2}\\
	 &\leq (1+8d)^{n/2}\nu_{n}(x+\eta)
	\end{align*}
  and thus $(\omega.1)$. Furthermore, we derive $(\omega.2)$ and $(\omega.3)$ from
  \[
   \nu_{n}(x)=\frac{1}{(1+|x|^{2})^{d}}\nu_{n+2d}(x)=\psi_{n,d}(x)\nu_{n+2d}(x).
  \]
  \item [(v)] This follows directly from the choice $r_{k}(x):=\min(\D^{\infty}_{k+1},1)/2$, $x\in\Omega_{k}$, and $\psi_{n}:=1$.
\end{enumerate}
\end{proof}

Referring to \prettyref{rem:cond_sets} and \prettyref{ex:families_of_weights_1}, condition $(s.1)$ is fulfilled in (2) and (3). 
Condition $(s.2)$ is always fulfilled if 
$\Omega:=\Omega_{n}:=\R^{d}$ for all $n\in\N$. Further examples 
of sets $\Omega$ and $(\Omega_{n})_{n\in\N}$ satisfying $(s.2)$ and the conditions of the preceding example 
are given below.

\begin{exa}\label{ex:families_of_sets}
The following non-empty open sets $\Omega\subset\R^{d}$ and families $(\Omega_{n})_{n\in\N}$ 
of sets fulfil $\Omega=\bigcup_{n\in\N} \Omega_{n}$, $\Omega_{n}\subset\Omega_{n+1}$ for all $n\in\N$ and 
$0<\D^{\infty}_{n+1}<\infty$ for all $n\geq N$ for some $N\in\N$:
\begin{enumerate}
\item [(i)] $\Omega$ arbitrary, $\Omega_{n}:=\mathring K_{n}$ where $K_{n}\subset\mathring K_{n+1}$ is a compact exhaustion of $\Omega$.
\item [(ii)] $\Omega$ with $\partial \Omega\neq\varnothing$, 
$\Omega_{n}:=\{x=(x_{i})\in \Omega\;|\;\forall\;i\in I:\;|x_{i}|< n\;\text{and}\;\d^{\|\cdot\|}_{\partial\Omega}(x)> 1/n\}$ 
where $I\subset\{1,\ldots,d\}$ and $\d^{\|\cdot\|}_{\partial\Omega}(x):=\inf\{\|x-\zeta\|\;|\;\zeta\in\partial\Omega\}$, 
$x\in\Omega$, for some norm $\|\cdot\|$ on $\R^{d}$.
\item [(iii)] $\Omega:=\R^{d}$, $\Omega_{n}:=\{x=(x_{i})\in \Omega\;|\;\forall\;i\in I:\;|x_{i}|< n\}$ 
where $I\subset\{1,\ldots,d\}$.
\end{enumerate}
The same is true and additionally $(s.2)$ holds if the family $(\Omega_{n})_{n\in\N}$ in the preceding examples 
is replaced by the sequence of closures $(\overline{\Omega}_{n})_{n\in\N}$.
\end{exa}

Choosing $\Omega$ and $(\Omega_{n})_{n\in\N}$ from \prettyref{ex:families_of_sets} (ii) or (iii) with $I\subsetneq\{1,\ldots,d\}$ and 
\[
\mu\colon \Omega\to [0,\infty),\; \mu(x):=|(x_{i})_{i\in I_{0}}|^{\gamma},
\]
for $I_{0}:=\{1,\ldots,d\}\setminus I$ and $0<\gamma\leq 1$, we see that we are 
in the situation of \prettyref{ex:families_of_weights_1} (ii)
with $c_{n}:=(\sum_{i\in I}n)^{\gamma}$ since $\mu$ is $\gamma$-H\"{o}lder continuous.
In \prettyref{ex:families_of_weights_1} (iv) we have $\mathcal{EV}(\R^{d})=\mathcal{S}(\R^{d})$, the Schwartz space, and in (v) 
with \prettyref{ex:families_of_sets} (i) $\mathcal{EV}(\Omega)=\mathcal{C}^{\infty}(\Omega,\K)$ equipped with 
the usual topology of uniform convergence of all partial derivatives on compact subsets of $\Omega$.

Conditions like $(\omega.1)$-$(\omega.3)$ are also considered in \cite[(5.1), p.\ 41]{L3} for holomorphic functions 
on $\Omega:=\C\setminus [0,\infty)$ with
\[
\Omega_{n}:=\{z\in\C\;|\;|\im(z)|\leq n\}\setminus\{z\in\C\;|\;\re(z)>-1/n,\;|\im(z)|<1/n\},
\]
$\psi_{n}(z):=(1+|z|^{2})^{-2}$ and continuous functions $\mu_{n}\colon\Omega\to[0,\infty)$ such that $\mu_{n}\leq\mu_{n+1}$ 
defining $\nu_{n}(z):=e^{\mu_{n}(z)}$ for every $n\in\N$. 
The condition \cite[(3.2), p.\ 37]{L3} implies (with $x=0$ there and the choice 
$r_{k}(z):=\min(\D^{\infty}_{k+1},\gamma,1)/2$, $z\in\Omega_{k}$, with $\gamma$ from that condition) 
that $(\omega.1)$ and $(\omega.3)$ hold. For $\Omega:=\C\setminus\R$ a similar construction is used.

If $\Omega:=\Omega_{n}:=\R^{d}$ for all $n\in\N$, then condition $(s.3)$ is also fulfilled 
in the corresponding examples of \prettyref{ex:families_of_weights_1} since the constructed $r_{k}$ 
are continuous for every $k\in\N$. Next, we consider an example where $\Omega\neq\R^{d}$.

\begin{exa}\label{ex:families_of_weights_2} 
Let $\Omega\subset\R^{d}$ be a non-empty bounded open set and $\Omega_{n}:=\Omega$ for all $n\in\N$.
The family $\mathcal{V}:=(\nu_{n})_{n\in\N}$ of positive continuous functions on $\Omega$ given by 
\[
 \nu_{n}\colon\Omega\to (0,\infty),\;\nu_{n}(x):=\max_{1\leq j\leq n}\d^{\infty}_{\partial\Omega}(x)^{-j},
\]
fulfils $\nu_{n}\leq \nu_{n+1}$ for all $n\in\N$ and $(\omega.1)$-$(\omega.3)$. Moreover, $(s.3)$ is satisfied.
\end{exa}
\begin{proof}
 We choose $r(x):=r_{k}(x):=\min(\d^{\infty}_{\partial\Omega}(x)/2,1)$, $x\in\Omega_{k}=\Omega$, for every $k\in\N$. 
 For $x\in\Omega$ and $\|\zeta\|_{\infty},\|\eta\|_{\infty}\leq r(x)$ we have 
 \[
  \d^{\infty}_{\partial\Omega}(x+\eta)\leq \d^{\infty}_{\partial\Omega}(x+\zeta)+\|\eta-\zeta\|_{\infty}
  \leq \d^{\infty}_{\partial\Omega}(x+\zeta)+2r(x)\leq \d^{\infty}_{\partial\Omega}(x+\zeta)+\d^{\infty}_{\partial\Omega}(x)
 \]
 and 
 \[
  \d^{\infty}_{\partial\Omega}(x)\leq \d^{\infty}_{\partial\Omega}(x+\zeta)+\|\zeta\|_{\infty}
  \leq \d^{\infty}_{\partial\Omega}(x+\zeta)+\d^{\infty}_{\partial\Omega}(x)/2,
 \]
 implying 
 \[
  \d^{\infty}_{\partial\Omega}(x+\eta)\leq 3\d^{\infty}_{\partial\Omega}(x+\zeta).
 \]
 We deduce that $\nu_{n}(x+\zeta)\leq 3^{n}\nu_{n}(x+\eta)$ and conclude that $(\omega.1)$ holds. Furthermore, we observe that 
 \[
  \nu_{n}(x)=r(x)\nu_{n}(x)/r(x)\leq 2r(x)\nu_{n+1}(x)
 \]
 holds for all $x\in\Omega$, yielding $(\omega.3)$. We note that $r$ is continuously extendable on $\overline{\Omega}$ 
 by setting $r:=0$ on $\partial\Omega$. 
 Thus $(s.3)$ and $(\omega.2)$ are valid since $\psi_{n}:=r\in L^{1}(\Omega)$ as $\Omega$ is bounded.
\end{proof}

Now, we modify a partition of unity constructed in 
\cite[1.4 Lemma, p.\ 207]{L4} for the situation of case $(s.3)$, 
which itself is a modification of the partition of unity from \cite[Lemma 1.4.9, p.\ 29]{H1}.

\begin{lem}\label{lem:partition_unity}
Let $n\in\N$, $\mathcal{V}$ fulfil $(\omega.3)$ and one of the conditions $(s.1)$-$(s.3)$ be satisfied.
There are $\mathcal{K}\subset\N$ and a sequence $(z_{k})_{k\in \mathcal{K}}$ in $\Omega_{n}$ 
with $z_{k}\neq z_{j}$, $k\neq j$, such that the following holds:
\begin{enumerate}
 \item [(p.1)] The balls
 \[
  b_{k}:=\{\zeta\in\R^{d}\;|\;\|\zeta-z_{k}\|_{\infty}<r_{n}(z_{k})/2\},\quad k\in \mathcal{K},
 \]
 form an open covering of $\Omega_{n}$, and any $x\in\Omega_{n}$ is contained in at most $(8/r_{n,2}(x))^{d}$ different balls 
 \[
 B_{k}:=\{\zeta\in\R^{d}\;|\;\|\zeta-z_{k}\|_{\infty}<r_{n}(z_{k})\},\quad k\in \mathcal{K},
 \]
 such that 
 \[
 \Omega_{n}\subset \bigcup_{k\in \mathcal{K}}b_{k}\subset\bigcup_{k\in \mathcal{K}}B_{k}\subset \Omega_{n+1}.
 \]
 \begin{center}
\begin{minipage}{\linewidth}
\centering
\begin{tikzpicture}[scale=0.65]
\draw[->](-0.2,0)--(7,0) node[right] {$x_{1}$} coordinate (x axis);
\draw[->](0,-4.5)--(0,4.5) node[above] {$x_{2}$} coordinate (y axis);
\draw[dashed,thick] (7,2) -- (6,2) arc (90:270:20mm) -- (7,-2);
\draw[dashed,thick] (7,4.5) -- (6,4.5) arc (90:270:45mm) -- (7,-4.5);
 
\node (A) at (6.5,-1.5) {$\Omega_{n}$};
\node (B) at (6.5,-3.25) {$\Omega_{n+1}$}; 

\draw[thick,gray] (5.5,0.75) rectangle (7,2.25);
\draw[thick,gray] (5,1.25) rectangle (5.75,2);
\draw[thick,gray] (4.35,1) rectangle (5.15,1.8);
\draw[thick,red] (4.85,0.5) rectangle (5.85,1.5);
\draw[thick,blue] (5.25,-0.25) rectangle (6.5,1);
\draw[thick,green!60!black] (4.45,-0.55) rectangle (5.65,0.65);
\draw[thick,gray] (4,0.15) rectangle (5,1.15);

\draw[thick,dotted] (3.5,-0.8) rectangle (4.6,0.3);
\draw[thick,dotted] (3.95,-1.2) rectangle (4.85,-0.3);
\draw[thick,dotted] (4.7,-1.3) rectangle (6,0);
\draw[thick,dotted] (5.5,-2.35) rectangle (7,-0.85);

\draw[thick,dotted] (4.3,-1.8) rectangle (5.1,-1);
\draw[thick,dotted] (4.9,-2.15) rectangle (5.9,-1.15);
\draw[thick,dotted] (7,0.5)--(5.75,0.5)--(5.75,-0.95)--(7,-0.95);
\draw[thick,dotted] (7,0.9)--(6.3,0.9)--(6.3,0.1)--(7,0.1);

\def\mypath{
(15.125,3)--(12.75,3)--(12.75,2.2)--(11.95,2.2)--(11.95,1.65)--(11.5,1.65)--(11.5,0.85)--(10.95,0.85)--(10.95,-1.35)--(11.5,-1.35)
--(11.5,-1.65)--(11.9,-1.65)--(11.9,-2.2)--(12.4,-2.2)--(12.4,-2.65)--(12.75,-2.65)--(12.75,-3.1)--(15.125,-3.1)}
\fill[fill=black!10,draw=black] \mypath;
 
\draw[->](7.8,0)--(15.125,0) node[right] {$x_{1}$} coordinate (x axis);
\draw[->](8,-4.5)--(8,4.5) node[above] {$x_{2}$} coordinate (y axis);

\draw[dashed,thick] (15.125,4.5) -- (14,4.5) arc (90:270:45mm) -- (15.125,-4.5);
\draw[dashed,thick] (15.125,2) -- (14,2) arc (90:270:20mm) -- (15.125,-2);

\node[scale=0.7] (C) at (14.25,2.4) {$\bigcup\limits_{k\in \mathcal{K}}B_{k}$};

\draw[thick,red,densely dotted] (12.85,0.5) rectangle (13.85,1.5);
\draw[thick,blue,densely dotted] (13.25,-0.25) rectangle (14.5,1);
\draw[thick,green!60!black,densely dotted] (12.45,-0.55) rectangle (13.65,0.65);

\draw[thick,red] (12.35,0) rectangle (14.35,2);
\draw[thick,blue] (12.625,-0.875) rectangle (15.125,1.625);
\draw[thick,green!60!black] (11.85,-1.15) rectangle (14.25,1.25);
 
\end{tikzpicture}
\end{minipage}
\captionsetup{type=figure}
\caption{Cover of $\Omega_{n}$ with squares of the form $b_{k}$ resp.\ $B_{k}$ contained in $\Omega_{n+1}$, $d=2$}
\end{center}
 \item [(p.2)] The set $M_{k}:=\{m\in \mathcal{K}\;|\;B_{m}\cap B_{k}\neq \varnothing\}$ contains 
 at most $(8/r_{n,3}(z_{k}))^{d}$ elements for every $k\in \mathcal{K}$.
 \item [(p.3)] There is a $\mathcal{C}^{\infty}$-partition of unity $\{h_{k}\;|\;k\in \mathcal{K}\}$ subordinate to 
 $\{B_{k}\;|\;k\in \mathcal{K}\}$ such that for every $\alpha\in\N_{0}^{d}$ there is $\widetilde{C}_{\alpha}>0$ such that 
 for every $k \in \mathcal{K}$
 \[
 |\partial^{\alpha}h_{k}|\leq \widetilde{C}_{\alpha}\bigl(1/r_{n,3}(z_{k})\bigr)^{d+|\alpha|}.
 \]
 \item [(p.4)] Let $\mathcal{V}$ fulfil $(\omega.1)$. For every $m,j,p\in\N$ there is $D>0$ such that for all $k\in\mathcal{K}$
\[
 \nu_{m}(x)\leq D r_{n,j}(z_{k})^{p}\nu_{J_{1}P_{3}J_{1}I_{1}(m)}(z_{k}),\quad x\in B_{k},
\]
where $J_{1}$ is the $j$-fold composition of $I_{1}$ and $P_{3}$ the $p$-fold composition of $I_{3}$.
\end{enumerate}
\end{lem}
\begin{proof}
Let $Z$ be the set of subsets $X\subset\Omega_{n}$ such that
\begin{equation}\label{eq:partition_unity1}
\forall\;z,y\in X,\,z\neq y:\; \|z-y\|_{\infty}\geq r_{n,1}(y)/2
\end{equation}
and let $Z$ be equipped with the inclusion $\subset$ as partial order. Every chain $A\subset Z$ has an upper bound in $Z$
given by $\bigcup_{X\in A} X$, yielding that $(Z,\subset)$ is inductively ordered. 
Due to Zorn's Lemma, there is a maximal element $X_{0}\in Z$. 
We note that $X_{0}\subset\Omega_{n}$ is a discrete set since $r_{n,1}>0$ on $\Omega_{n}$ and 
$\{z\in X_{0}\;|\;\|z-y\|_{\infty}< r_{n,1}(y)/2\}=\{y\}$ for every $y\in X_{0}$ by \eqref{eq:partition_unity1}. 
Further, a discrete subset of $\Omega_{n}\subset\R^{d}$ cannot be uncountable by \cite[4.1.15 Theorem, p.\ 255]{engelking}.
Hence there is some $l\in\N\cup\{\infty\}$, a set $\mathcal{K}:=\{k\in\N\;|\;k\leq l\}$ and 
a maximal sequence $(z_{k})_{k\in\mathcal{K}}$ in $\Omega_{n}$ such that
\begin{equation}\label{thm13.0.0.0}
 \|z_{k}-z_{j}\|_{\infty}\geq r_{n,1}(z_{j})/2, \quad k\neq j,
\end{equation}
which also implies $z_{k}\neq z_{j}$ for $k\neq j$. 

$(p.1)(i)$ Let $x\in\Omega_{n}$. If there is no $k\in\mathcal{K}$ such that 
$x=z_{k}$, then there is $j\in\mathcal{K}$ such that 
\[
 \|x-z_{j}\|_{\infty}< r_{n,1}(z_{j})/2\underset{\eta=z=z_{j}}{\leq} r_{n}(z_{j})/2
\]
or 
\[
 \|x-z_{j}\|_{\infty}< \underbrace{r_{n,1}(x)/2}_{\leq r_{n}(x)}\underset{z=x,\eta=z_{j}}{\leq} r_{n}(z_{j})/2
\]
due to the maximality of $(z_{k})_{k\in\mathcal{K}}$. Thus the balls $b_{k}$, $k\in\mathcal{K}$, cover $\Omega_{n}$. 
The condition $r_{n}<\d^{\infty}_{n+1}$ on $\Omega_{n}$ guarantees that $B_{k}\subset \Omega_{n+1}$ for every $k\in\mathcal{K}$.

$(p.1)(ii)$ Let $x\in\Omega_{n}$ and set 
$\beta_{k}:=\{\zeta\in\R^{d}\;|\;\|\zeta-z_{k}\|_{\infty}<r_{n,2}(x)/4\}$ for $k\in\mathcal{K}$. 
The number of balls $\{B_{k}\;|\;\exists\; k\in\mathcal{K}:\;x\in B_{k}\}$ coincides with the number of balls 
$\beta:=\{\beta_{k}\;|\;\exists\; k\in\mathcal{K}:\;x\in B_{k}\}$. 
We observe that the elements of $\beta$ are disjoint. Otherwise, there are $k,j\in\mathcal{K}$, $k\neq j$, such that 
$x\in B_{k}\cap B_{j}$ and there is $\zeta\in\beta_{k}\cap\beta_{j}$, which yields the contradiction
\[
 r_{n,2}(x)/2\underset{\eta=z_{j},z=x}{\leq}r_{n,1}(z_{j})/2\underset{\eqref{thm13.0.0.0}}{\leq}\|z_{k}-z_{j}\|_{\infty}
 \leq \|z_{k}-\zeta\|_{\infty}+\|\zeta-z_{j}\|_{\infty}<r_{n,2}(x)/2.
\]
Moreover, $\beta_{k}\subset\{\zeta\in\R^{d}\;|\;\|\zeta-x\|_{\infty}\leq 2\}=:B$ for every $\beta_{k}\in\beta$. Indeed, 
if $\zeta\in\beta_{k}$ and $x\in B_{k}$, then 
\[
 \|\zeta-x\|_{\infty}\leq\|\zeta-z_{k}\|_{\infty}+\|z_{k}-x\|_{\infty}<r_{n,2}(x)/4+r_{n}(z_{k})\leq (1/4)+1<2.
\]
The ball $B$ contains at most $(2/(r_{n,2}(x)/4))^{d}=(8/r_{n,2}(x))^{d}$ disjoint balls of the form $\beta_{k}$. 
Hence the number of elements of $\beta$ is at most $(8/r_{n,2}(x))^{d}$.

$(p.2)$ Like in $(p.1)(ii)$ we conclude that the balls
$\widetilde{\beta}_{m}:=\{\zeta\in\R^{d}\;|\;\|\zeta-z_{m}\|_{\infty}<r_{n,3}(z_{k})/4\}$ 
are disjoint and contained in $B$ for $m\in M_{k}$, which implies $(p.2)$.

$(p.3)$ First, we want to use the cut-off function from \cite[Theorem 1.4.2, p.\ 25-26]{H1}. 
Let $k\in\mathcal{K}$. We set $X:=B_{k}$, $K:=\overline{b_{k}}$ and choose a positive decreasing sequence 
$(w_{j})_{j\in\mathcal{K}}$ such that $\sum_{j\in\mathcal{K}}w_{j}=1$, e.g.\ 
$w_{j}:=1/2^{j}$ if $\mathcal{K}=\N$ or $w_{j}:=1/l$ if the number $l$ of elements of $\mathcal{K}$ is finite.   
Then 
\[
\delta:=\inf\{\|x-y\|_{\infty}\;|\;x\in\R^{d}\setminus X,\,y\in K\}=r_{n}(z_{k})/2
\]
and with the choice of $d_{j}:=w_{j}r_{n}(z_{k})/3$, $j\in\mathcal{K}$, we obtain 
$\sum_{j\in\mathcal{K}}d_{j}=r_{n}(z_{k})/3<\delta$. 
Moreover, we observe that for every function $\varphi\in\mathcal{C}^{\infty}(X,\R)$ and 
$\alpha=(\alpha_{1},\ldots,\alpha_{d})\in\N_{0}^{d}$ the relation 
\[
\partial^{\alpha}\varphi(x)=\varphi^{(|\alpha|)}(x;\underbrace{e_{1},\ldots,e_{1}}_{\alpha_{1}-\text{times}},
\underbrace{e_{2},\ldots,e_{2}}_{\alpha_{2}-\text{times}},\ldots,\underbrace{e_{d},\ldots,e_{d}}_{\alpha_{d}-\text{times}}),
\quad x\in X,
\]
between the $\alpha$-th partial derivative $\partial^{\alpha}\varphi$ and the differential $\varphi^{(|\alpha|)}$ of order 
$|\alpha|$ holds where $e_{n}$, $1\leq n\leq d$, denotes the $n$-th unit vector in $\R^{d}$. 
Due to \cite[Theorem 1.4.2, p.\ 25-26]{H1}, there exists $\varphi_{k}\in\mathcal{C}^{\infty}_{c}(B_{k},\R)$, i.e.\ 
$\varphi_{k}$ is smooth on $B_{k}$ and its support $\operatorname{supp}\varphi_{k}$ is compact,
such that $0\leq \varphi_{k}\leq 1$, $\varphi_{k}=1$ on a neighbourhood of $\overline{b_{k}}$ and
\[
|\partial^{\alpha}\varphi_{k}(x)|
\leq \frac{c^{|\alpha|}}{d_{1}\ldots d_{|\alpha|}}
=\frac{(3c)^{|\alpha|}}{w_{1}\ldots w_{|\alpha|}}\frac{1}{r_{n}(z_{k})^{|\alpha|}},\quad x\in B_{k},\;\alpha\in\N_{0}^{d}\setminus\{0\},
\]
where $c>0$ only depends on the dimension of $\R^{d}$ (and not on $k$ or $\alpha$). Furthermore, we set 
$\frac{(3c)^{|\alpha|}}{w_{1}\ldots w_{|\alpha|}}:=1$ if $\alpha=0$, and note that 
\[
|\partial^{0}\varphi_{k}(x)|=\varphi_{k}(x)\leq 1= \frac{(3c)^{|0|}}{w_{1}\ldots w_{|0|}}\frac{1}{r_{n}(z_{k})^{|0|}},\quad x\in B_{k}.
\]
\begin{center}
\begin{minipage}{\linewidth}
\centering
\begin{tikzpicture}[scale=0.75]
\draw[->](0.8,1)--(6.5,1) node[right] {$x_{1}$} coordinate (x axis);
\draw[->](1,0.8)--(1,6.5) node[above] {$x_{2}$} coordinate (y axis); 
\begin{scope}
\clip(3,2.7)--(5,2.7) arc (270:360:3mm) -- (5.3,5) arc (0:90:3mm) -- (3,5.3) arc (90:180:3.mm) -- (2.7,3) arc (180:270:3mm);
\foreach \x in {-0.1,0.3,...,5.1}
\draw[green!60!black,line width=0.2mm](\x, 2)--+(4,4);
\end{scope}
\path[fill=white] (3.5,3.8) rectangle (4.45,4.2);
\draw (2,2) rectangle (6,6);
\draw (3,3) rectangle (5,5);
\draw (3,2.7)--(5,2.7) arc (270:360:3mm) -- (5.3,5) arc (0:90:3mm) -- (3,5.3) arc (90:180:3.mm) -- (2.7,3) arc (180:270:3mm);
\draw[blue,rounded corners=5pt] (3,2.35)--(4.5,2.25)--(5.1,2.4)--(5.75,3)--(5.5,4)--(5.65,5.55)--(3.3,5.85)--
(2.3,5.5)--(2.5,4)--(2.2,2.75)--(2.7,2.3)--(3,2.35);
\node[text=blue,scale=0.7] (A) at (3.5,5.5) {$\operatorname{supp}\varphi_{k}$};
\node[scale=0.7,text=green!20!black] (B) at (4,4) {$\varphi_{k}\equiv 1$};
\node (C) at (4.1,3.3) {$b_{k}$};
\node (D) at (5.6,2.3) {$B_{k}$};
\draw[red,thick,densely dotted] (3.2,3)--(3.2,2) node [anchor=north] {$r_{n}(z_{k})/2$};
\end{tikzpicture}
\end{minipage}
\captionsetup{type=figure}
\caption{Support of $\varphi_{k}$, $d=2$}
\end{center}
We denote by $q_{k}$ the number of elements of $M_{k}\cap\{1,\ldots,k-1\}$, enumerate this set
by $M_{k}\cap\{1,\ldots,k-1\}=\{m_{j}\;|\;1\leq j\leq q_{k}\}$, 
set $h_{k}:=\varphi_{k}(1-\varphi_{1})\cdots(1-\varphi_{k-1})$ for $k\geq 2$ and note that 
\[
 h_{k}=\varphi_{k}\prod_{j=1}^{q_{k}}(1-\varphi_{m_{j}}).
\]
Further, we define $h_{1}:=\varphi_{1}$. Let $\alpha\in\N_{0}^{d}\setminus\{0\}$. Since
\[
 r_{n}(z_{m})\geq r_{n,1}(x)\geq r_{n,2}(z_{k})\geq r_{n,3}(z_{k})
\]
for $x\in B_{m}\cap B_{k}$, where the first inequality follows from $\|x-z_{m}\|_{\infty}\leq r_{n}(z_{m})$ and 
the second from $\|x-z_{k}\|_{\infty}\leq r_{n}(z_{k})$, we get by the Leibniz rule for partial derivatives
\begin{flalign*}
 &\hspace{0.35cm}|\partial^{\alpha}h_{k}|\\
 &=|\partial^{\alpha}(\varphi_{k}\prod_{j=1}^{q_{k}}(1-\varphi_{m_{j}}))|\\
 &=\bigl|\sum_{\gamma_{1}\leq \alpha}\sum_{\gamma_{2}\leq\gamma_{1}}\cdots\sum_{\gamma_{q_{k}}\leq\gamma_{q_{k}-1}}
 \dbinom{\alpha}{\gamma_{1}}\dbinom{\gamma_{1}}{\gamma_{2}}\cdots\dbinom{\gamma_{q_{k}-1}}{\gamma_{q_{k}}}
 \partial^{\alpha-\gamma_{1}}\varphi_{k}\partial^{\gamma_{1}-\gamma_{2}}(1-\varphi_{m_{1}})\cdot\ldots\\ 
 &\hspace{6cm}\cdot
 \partial^{\gamma_{q_{k}-1}-\gamma_{q_{k}}}(1-\varphi_{m_{q_{k}-1}})\partial^{\gamma_{q_{k}}}(1-\varphi_{m_{q_{k}}})\bigr|\\
 &\leq\sum_{\gamma_{1}\leq\alpha}\sum_{\gamma_{2}\leq\gamma_{1}}\cdots\sum_{\gamma_{q_{k}}\leq\gamma_{q_{k}-1}}
 \dbinom{\alpha}{\gamma_{1}}\dbinom{\gamma_{1}}{\gamma_{2}}\cdots\dbinom{\gamma_{q_{k}-1}}{\gamma_{q_{k}}}
 |\partial^{\alpha-\gamma_{1}}\varphi_{k}||\partial^{\gamma_{1}-\gamma_{2}}\varphi_{m_{1}}|\cdot\ldots\\ 
 &\hspace{6cm}\cdot |\partial^{\gamma_{q_{k}-1}-\gamma_{q_{k}}}\varphi_{m_{q_{k}-1}}||\partial^{\gamma_{q_{k}}}\varphi_{m_{q_{k}}}|\\
 &\leq \sum_{\gamma_{1}\leq \alpha}\sum_{\gamma_{2}\leq\gamma_{1}}\cdots\sum_{\gamma_{q_{k}}\leq\gamma_{q_{k}-1}}\alpha ! 
 \frac{\frac{(3c)^{|\alpha-\gamma_{1}|}}{w_{1}\ldots w_{|\alpha-\gamma_{1}|}}}{r_{n}(z_{k})^{|\alpha-\gamma_{1}|}}
 \frac{\frac{(3c)^{|\gamma_{1}-\gamma_{2}|}}{w_{1}\ldots w_{|\gamma_{1}-\gamma_{2}|}}}{r_{n}(z_{m_{1}})^{|\gamma_{1}-\gamma_{2}|}}
 \cdot\ldots \\
 &\hspace{6cm}\cdot\frac{\frac{(3c)^{|\gamma_{q_{k}-1}-\gamma_{q_{k}}|}}{w_{1}\ldots w_{|\gamma_{q_{k}-1}-\gamma_{q_{k}}|}}}
 {r_{n}(z_{m_{q_{k}-1}})^{|\gamma_{q_{k}-1}-\gamma_{q_{k}}|}}
 \frac{\frac{(3c)^{|\gamma_{q_{k}}|}}{w_{1}\ldots w_{|\gamma_{q_{k}}|}}}
 {r_{n}(z_{m_{q_{k}}})^{|\gamma_{q_{k}}|}}\\
 &\leq \sum_{\gamma_{1}\leq \alpha}\sum_{\gamma_{2}\leq\gamma_{1}}\cdots\sum_{\gamma_{q_{k}}\leq\gamma_{q_{k}-1}}\alpha ! 
 (3c)^{|\alpha|}\frac{1}{r_{n,3}(z_{k})^{|\alpha|}}\\
 &\hspace{3cm}\cdot\underbrace{\frac{1}{w_{1}\ldots w_{|\alpha-\gamma_{1}|}}\frac{1}{w_{1}\ldots w_{|\gamma_{1}-\gamma_{2}|}}
 \cdot\ldots\cdot\frac{1}{w_{1}\ldots w_{|\gamma_{q_{k}-1}-\gamma_{q_{k}}|}}
 \frac{1}{w_{1}\ldots w_{|\gamma_{q_{k}}|}}}_{\leq\frac{1}{w_{1}\ldots w_{|\alpha|}}\quad (\star)}\\
 &\underset{\mathclap{(p.2)}}{\leq}\; 8^{d}(d^{\alpha_{1}}+d^{\alpha_{2}}+\ldots+d^{\alpha_{d}})\alpha ! 
 (3c)^{|\alpha|}\frac{1}{w_{1}\ldots w_{|\alpha|}}
 \frac{1}{r_{n,3}(z_{k})^{d+|\alpha|}}
\end{flalign*}
where we used for $(\star)$ that $(w_{j})_{j\in\mathcal{K}}$ is decreasing and
$|\alpha-\gamma_{1}|+|\gamma_{1}-\gamma_{2}|+\ldots+|\gamma_{q_{k}-1}-\gamma_{q_{k}}|+|\gamma_{q_{k}}|=|\alpha|$ 
and in the last inequality that the number of sums is at most $(8/r_{n,3}(z_{k}))^{d}$ 
and the number of summands in each sum at most $d^{\alpha_{1}}+d^{\alpha_{2}}+\ldots+d^{\alpha_{d}}$. 
If $\alpha=0$, then 
\[
|\partial^{0}h_{k}|=|h_{k}|\leq 1\leq\frac{1}{r_{n,3}(z_{k})^{d}}
\]
because $r_{n,3}(z_{k})\leq 1$.
Observing that $\sum_{k\in\mathcal{K}}h_{k}=1$ on $\Omega_{n}$ as $(\sum_{k\in\mathcal{K}}h_{k})-1=-\prod_{k\in\mathcal{K}}(1-\varphi_{k})=0$ on 
$\bigcup_{k\in\mathcal{K}}b_{k}$, we proved $(p.3)$.

$(p.4)$ First, we claim that for every $j\in\N$ it holds that
\begin{equation}\label{eq:induction}
\forall\;m,p\in\N\;\exists\;D>0\;\forall\;z\in\Omega_{n}:\; \nu_{m}(z)\leq Dr_{n,j}(z)^{p}\nu_{J_{1}P_{3}J_{1}(m)}(z).
\end{equation}
This follows by induction on $j$. The base case is
\begin{align*}
\frac{\nu_{m}(z)}{r_{n,1}(z)^{p}}
&=\sup\Bigl\{\frac{\nu_{m}(z)}{r_{n}(\eta)^{p}}\;|\;\eta\in\Omega_{n}:\;\|\eta-z\|_{\infty}\leq r_{n}(\eta)\;\text{or}\;
\|\eta-z\|_{\infty}\leq r_{n}(z)\Bigr\}\\
&\underset{\mathclap{(\omega.1)}}{\leq}A_{1}\sup\Bigl\{\frac{\nu_{I_{1}}(\eta)}{r_{n}(\eta)^{p}}\;|\;\eta\in\Omega_{n}:\;
\|\eta-z\|_{\infty}\leq r_{n}(\eta)\;\text{or}\;\|\eta-z\|_{\infty}\leq r_{n}(z)\Bigr\}\\
&\underset{\mathclap{(\omega.3)}}{\leq}A_{1}A_{3}(I_{1})A_{3}(I_{31})\cdots A_{3}((P_{3}-1)I_{1})\\
&\qquad\sup\bigl\{\nu_{P_{3}I_{1}}(\eta)\;|\;\eta\in\Omega_{n}:\;\|\eta-z\|_{\infty}\leq r_{n}(\eta)\;\text{or}\;
\|\eta-z\|_{\infty}\leq r_{n}(z)\bigr\}\\
&\underset{\mathclap{(\omega.1)}}{\leq}D_{0}A_{1}(P_{3}I_{1})\nu_{I_{1}P_{3}I_{1}}(z)
\end{align*}
where $D_{0}:=A_{1}A_{3}(I_{1})A_{3}(I_{31})\cdots A_{3}((P_{3}-1)I_{1})$ and $P_{3}-1$ is the 
$(p-1)$-fold composition of $I_{3}$.
Let us suppose that \eqref{eq:induction} holds for some $j\in\N$. Then we have
\begin{align*}
\frac{\nu_{m}(z)}{r_{n,j+1}(z)^{p}}
 &=\sup\Bigl\{\frac{\nu_{m}(z)}{r_{n,j}(\eta)^{p}}\;|\;\eta\in\Omega_{n}:\;\|\eta-z\|_{\infty}\leq r_{n}(\eta)\;\text{or}\;
 \|\eta-z\|_{\infty}\leq r_{n}(z)\Bigr\}\\
 &\underset{\mathclap{(\omega.1)}}{\leq} A_{1}
 \sup\Bigl\{\frac{\nu_{I_{1}}(\eta)}{r_{n,j}(\eta)^{p}}\;|\;\eta\in\Omega_{n}:\;\|\eta-z\|_{\infty}\leq r_{n}(\eta)\;\text{or}\;
 \|\eta-z\|_{\infty}\leq r_{n}(z)\Bigr\}\\
 &\leq A_{1}D\sup\bigl\{\nu_{J_{1}P_{3}J_{1}I_{1}}(\eta)\;|\;\eta\in\Omega_{n}:\;\|\eta-z\|_{\infty}\leq r_{n}(\eta)\;\text{or}\;
 \|\eta-z\|_{\infty}\leq r_{n}(z)\bigr\}\\
 &\underset{\mathclap{(\omega.1)}}{\leq} A_{1}DA_{1}(J_{1}P_{3}(J_{1}+1))\nu_{(J_{1}+1)P_{3}(J_{1}+1)}(z)
\end{align*}
where $J_{1}+1$ is the $(j+1)$-fold composition of $I_{1}$, which proves the claim. 
Second, let $m,j,p\in\N$. Then there is $D>0$ such that for $k\in\mathcal{K}$ and $x\in B_{k}$
\[
 \frac{\nu_{m}(x)}{r_{n,j}(z_{k})^{p}}\underset{(\omega.1)}{\leq}
 A_{1}\frac{\nu_{I_{1}}(z_{k})}{r_{n,j}(z_{k})^{p}}
 \underset{\eqref{eq:induction}}{\leq} A_{1}D \nu_{J_{1}P_{3}J_{1}I_{1}(m)}(z_{k}).
\]
\end{proof}

In \cite[1.4 Lemma, p.\ 207]{L4} an estimate like in $(p.3)$ is only given for the gradient of $h_{k}$, which can be 
deduced from \cite[Theorem 1.4.1, p.\ 25]{H1} instead of \cite[Theorem 1.4.2, p.\ 25-26]{H1}. An estimate like $(p.4)$ can be found 
in \cite[p.\ 210]{L4}.

%% file: Weighted_diff3.tex
Now, we are able to state our main theorem on the nuclearity of the space $\mathcal{EV}(\Omega)$ 
which generalises \cite[Theorem 3.7, p.\ 23]{ich} and whose proof is inspired by a proof of nuclearity 
for the (non-weighted) space of $\mathcal{C}^{\infty}$-functions in \cite[Example 28.9 (1), p.\ 349-350]{meisevogt1997}.

\begin{thm}\label{thm:nuclear}
If $\mathcal{V}$ is a family of continuous weights satisfying $(\omega.1)$-$(\omega.3)$ 
and one of the conditions $(s.1)$-$(s.3)$ is fulfilled, 
then $\mathcal{EV}(\Omega)$ is nuclear.
\end{thm}
\begin{proof}
Let $f\in\mathcal{EV}(\Omega)$, $n\in\N$, $\alpha=(\alpha_{i})\in\N^{d}_{0}$, $|\alpha|\leq m$, and $(z_{k})_{k\in\mathcal{K}}$ 
the sequence from \prettyref{lem:partition_unity}. 
For every $k\in\mathcal{K}$ there exist $a_{i},\; c_{i}, \; i=1,\ldots,d,$ such that $B_{k}=\prod_{i=1}^{d}(a_{i},c_{i})$. 
We get (by induction) for $x=(x_{i})\in b_{k}$, $k\in\mathcal{K}$, with $x':=(x_{2},\ldots,x_{d})$ and 
$\alpha':=(\alpha_{2},\ldots,\alpha_{d})$ 
\begin{flalign}\label{thm13.0.1}
&\hspace{0.35cm} |\partial^{\alpha}(h_{k}f)(x)|\nonumber\\
&=|\partial^{\alpha}(h_{k}f)(x_{1},x')-
\underbrace{\partial^{\alpha}(h_{k}f)(a_{1},x')}_{\underset{(p.3)}{=}0}|\nonumber\\
&=\bigl|\int^{x_{1}}_{a_{1}}{\partial^{(\alpha_{1}+1,\alpha')}(h_{k}f)(\zeta_{0},x')\d\zeta_{0}}\bigr|\nonumber\\
&=\bigl|\int^{x_{1}}_{a_{1}}\int^{\zeta_{0}}_{a_{1}}\cdots\int^{\zeta_{m-\alpha_{1}-1}}_{a_{1}}{\partial^{(m+1,\alpha')}
(h_{k}f)(\zeta_{m-\alpha_{1}},x')\d\zeta_{m-\alpha_{1}}\ldots \d\zeta_{1}\d\zeta_{0}}\bigr|\nonumber\\
&\leq|x_{1}-a_{1}|^{m-\alpha_{1}}\int^{x_{1}}_{a_{1}}{|\partial^{(m+1,\alpha')}(h_{k}f)(\zeta_{1},x')|\d\zeta_{1}}\nonumber\\
&\leq\underbrace{|x_{1}-a_{1}|^{m-\alpha_{1}}\ldots |x_{d}-a_{d}|^{m-\alpha_{d}}}_{\leq(2r_{n}(z_{k}))^{dm-|\alpha|}}
\int^{x_{1}}_{a_{1}}\cdots\int^{x_{d}}_{a_{d}}{|\partial^{(m+1,\ldots,m+1)}(h_{k}f)(\zeta_{1},\ldots,\zeta_{d})|
\d\zeta_{d}\ldots \d\zeta_{1}}\nonumber\\
&\leq 2^{dm}\int_{B_{k}}{|\partial^{(m+1,\ldots,m+1)}(h_{k}f)(\zeta)|\d\zeta}
\end{flalign}
where we used Fubini's theorem in the last inequality.
Furthermore, we set $M(x):=\{k\in\mathcal{K}\;|\;x\in \operatorname{supp}h_{k}\}$ for $x\in\Omega_{n}$. 
For all $x\in\Omega_{n}$ there is $k\in\mathcal{K}$ such that $x\in B_{k}$ and hence the number of elements 
of $M(x)$ is not greater than the one of $M_{k}$ which is at most $(8/r_{n,3}(z_{k}))^{d}$ by $(p.2)$. 
We get for all $x\in \Omega_{n}$
\begin{align}\label{thm13.0.2}
|\partial^{\alpha}f(x)|
&\underset{(p.3)}{=}\bigl|\partial^{\alpha}(\sum_{k\in \mathcal{K}}h_{k}f)(x)\bigr|
\underset{(p.2)}{=}\bigl|\sum_{k\in \mathcal{K}}\partial^{\alpha}(h_{k}f)(x)\bigr|\nonumber\\
&\hspace{0.21cm}\leq\sum_{k\in \mathcal{K}}|\partial^{\alpha}(h_{k}f)(x)|
=\sum_{k\in M(x)}|\partial^{\alpha}(h_{k}f)(x)|\nonumber\\
&\hspace{0.12cm}\underset{\eqref{thm13.0.1}}{\leq}2^{dm}\sum_{k\in M(x)}
\int_{B_{k}}{|\partial^{(m+1,\ldots,m+1)}(h_{k}f)(\zeta)|\d\zeta}.
\end{align}
For $k\in\mathcal{K}$ we define 
\[
 Q_{k}:=\{\zeta\in\R^{d}\;|\;\|\zeta-z_{k}\|_{\infty}<r_{n,1}(z_{k})/8\}\subset b_{k}
\]
and consider the map
\[
\Phi_{k}\colon \R^{d}\to\R^{d}, \;
\Phi_{k}(\zeta):=8\frac{r_{n}(z_{k})}{r_{n,1}(z_{k})}\zeta-\Bigl(8\frac{r_{n}(z_{k})}{r_{n,1}(z_{k})}-1\Bigr)z_{k}.
\]
Then $\Phi_{k}$ is a $\mathcal{C}^{1}$-diffeomorphism, $\Phi_{k}^{(1)}(\zeta)=
\operatorname{diag}\bigl(8\frac{r_{n}(z_{k})}{r_{n,1}(z_{k})},\ldots,8\frac{r_{n}(z_{k})}{r_{n,1}(z_{k})}\bigr)$ 
as Jacobian matrix for $\zeta\in\R^{d}$ and $\Phi_{k}(Q_{k})=B_{k}$. Moreover, we obtain, via the chain rule, for all $\beta\in\N^{d}_{0}$
\begin{equation}\label{thm13.0.3}
\partial^{\beta}(h_{k}f)(\Phi_{k}(\zeta))=
\Bigl(8\frac{r_{n}(z_{k})}{r_{n,1}(z_{k})}\Bigr)^{-|\beta|}\partial^{\beta}(h_{k}f\circ\Phi_{k})(\zeta),\quad\zeta\in\R^{d}.
\end{equation}
Since
\begin{equation}\label{thm13.0.4}
\operatorname{supp}(h_{k}f\circ\Phi_{k})\subset\operatorname{supp}(h_{k}\circ\Phi_{k})
=\Phi_{k}^{-1}(\operatorname{supp}h_{k})
\subset Q_{k}
\end{equation}
for all $k\in\mathcal{K}$ by $(p.3)$ and the definition of $\Phi_{k}$, we get for $k,j\in\mathcal{K}$, $k\neq j$,
\[
[\operatorname{supp}( h_{k}\circ\Phi_{k})\cap\operatorname{supp}(h_{j}\circ\Phi_{j})]
\subset [Q_{k}\cap Q_{j}]\underset{\eqref{thm13.0.0.0}}{=}\varnothing
\]
and thus by \eqref{thm13.0.3}
\begin{equation}\label{thm13.0.5}
\operatorname{supp}([\partial^{\beta}(h_{k}f)]\circ\Phi_{k})\cap\operatorname{supp}([\partial^{\beta}(h_{j}f)]\circ\Phi_{j})
=\varnothing.
\end{equation}
 \begin{center}
	    \begin{minipage}{\linewidth}
	    \centering
	    \begin{tikzpicture}[scale=0.75]
	    \draw[->](1.5,2)--(7.5,2) node[right] {$x_{1}$} coordinate (x axis);
	    \draw[->](1.7,1.8)--(1.7,7.5) node[above] {$x_{2}$} coordinate (y axis);
	    \draw (4,3.5) rectangle (6.5,6);
	    \draw[blue,rounded corners=5pt] (4.25,3.25)--(5.75,3)--(7,3.75)--(6.6,4.75)--(6.9,6.3)--(4.55,6.6)--(3.55,6.25)
	    --(3.75,4.75)--(3.45,3.5)--(4.25,3.25); 
	    \node[scale=0.7] (A) at (4.3,3.75) {$b_{j}$};
	    \node[scale=0.7] (B) at (3.3,5.2) {$b_{k}$};
	    \node[text=blue,scale=0.7] (D) at (5.3,3.3) {$\operatorname{supp}h_{j}$};
	    \node[text=red,scale=0.7] (E) at (3.8,7.2) {$\operatorname{supp}h_{k}$};
	    \filldraw (5.25,4.75) circle [radius=2pt];
	    \filldraw (4,6) circle [radius=2pt];
	    
	    \node[scale=0.7] (I) at (4.3,6.15) {$z_{k}$};
	    \draw (3,5) rectangle (5,7);
	    \draw[red,rounded corners=5pt] (3,4.2)--(4.4,4.5)--(5.2,4.3)--(5.5,4.9)--(5.8,6.6)--(5.7,7.5)--(4.8,7.7)--(3.9,7.4)--(2.8,7.6)
	    --(2.4,6.1)--(2.6,5.2)--(2.3,4.5)--(3,4.2);
	    \node[scale=0.7] (H) at (4.9,4.75) {$z_{j}$};
	    
	    \draw [->,>=latex,arrow,bend left,thick] (7.3,5) to (9.35,5);
	    \node (F) at (8.2,5.7) {$\Phi^{-1}$};
	    \draw [->,>=latex,arrow,bend left,thick] (9.35,4) to (7.3,4);
	    \node (G) at (8.1,3.3) {$\Phi$};
	    
	    \draw[->](8.5,2)--(14,2) node[right] {$x_{1}$} coordinate (x axis);
	    \draw[->](8.7,1.8)--(8.7,7.5) node[above] {$x_{2}$} coordinate (y axis);
	    \draw (10.5,3.5) rectangle (13,6);
	    \draw[blue,rounded corners=1.67pt] (11.42,4.25)--(11.92,4.17)--(12.33,4.42)--(12.2,4.75)--(12.3,5.27)--(11.52,5.37)--(11.18,5.25)
	    --(11.25,4.75)--(11.15,4.33)--(11.42,4.25); 
	    \draw (9.5,5) rectangle (11.5,7);
	    \draw[red,rounded corners=1.67pt] (10.17,5.4)--(10.63,5.5)--(10.9,5.43)--(11,5.63)--(11.1,6.2)--(11.07,6.5)--(10.77,6.57)
	    --(10.47,6.47)--(10.1,6.53)--(9.97,6.03)--(10.03,5.73)--(9.93,5.5)--(10.17,5.4); 
	    \node[scale=0.7] (J) at (10.8,3.75) {$b_{j}$};
	    \node[scale=0.7] (K) at (9.8,5.2) {$b_{k}$};
	    \filldraw (11.75,4.75) circle [radius=2pt];
	    \filldraw (10.5,6) circle [radius=2pt];
	    \node[scale=0.7] (L) at  (11.45,4.75) {$z_{j}$};
	    \node[scale=0.7] (M) at (10.8,6.15) {$z_{k}$};
	    \node[text=red,scale=0.7] (N) at (12.3,6.5) {$\operatorname{supp}(h_{k}\circ \Phi_{k})$};
	    \node[text=blue,scale=0.7] (O) at (11.9,4) {$\operatorname{supp}(h_{j}\circ \Phi_{j})$};
	    \end{tikzpicture}
	    \end{minipage}
\captionsetup{type=figure}
\caption{$\operatorname{supp}(h_{k}\circ \Phi_{k})\cap\operatorname{supp}(h_{j}\circ \Phi_{j})=\varnothing$ 
for $j\in M_{k}$, $j\neq k$, $d=2$}
\end{center}
Applying the transformation formula to \eqref{thm13.0.2}, we obtain for all $x\in \Omega_{n}$ with 
$U:=\bigcup_{j\in\mathcal{K}}B_{j}$ and $p:=I_{11}D_{3}I_{11}(n)$, where $D_{3}$ is the $d$-fold composition of $I_{3}$,
\begin{flalign*}
&\quad\,|\partial^{\alpha}f(x)|\nu_{n}(x)\\
&\leq 2^{dm}\sum_{k\in M(x)}\int_{Q_{k}}{|\partial^{(m+1,\ldots,m+1)}(h_{k}f)(\Phi_{k}(\zeta))
|\Bigl(8\frac{r_{n}(z_{k})}{r_{n,1}(z_{k})}\Bigr)^{d}\d\zeta}\:\nu_{n}(x)\\
&\leq 16^{dm}\sum_{k\in M(x)}\int_{Q_{k}}{|\partial^{(m+1,\ldots,m+1)}(h_{k}f)(\Phi_{k}(\zeta))|
\frac{\nu_{n}(x)}{r_{n,1}(z_{k})^{d}}\d\zeta}\\
&\underset{\mathclap{(p.4)}}{\leq} 16^{dm}\sum_{k\in M(x)}\int_{Q_{k}}{|\partial^{(m+1,\ldots,m+1)}(h_{k}f)(\Phi_{k}(\zeta))|
D\nu_{I_{1}D_{3}I_{11}}(z_{k})\d\zeta}\\
&\underset{\mathclap{(\omega.1)}}{\leq}16^{dm}DA_{1}(I_{1}D_{3}I_{11})\sum_{k\in M(x)}\int_{Q_{k}}{|\partial^{(m+1,\ldots,m+1)}
(h_{k}f)(\Phi_{k}(\zeta))|\nu_{I_{11}D_{3}I_{11}}(\zeta)\d\zeta}\\
&\underset{\mathclap{\eqref{thm13.0.4}}}{=}C_{0}\sum_{k\in M(x)}\:\int_{U}{|\partial^{(m+1,\ldots,m+1)}
(h_{k}f)(\Phi_{k}(\zeta))|\nu_{p}(\zeta)\d\zeta}\\
&=C_{0}\int_{U}{\sum_{k\in M(x)}|\partial^{(m+1,\ldots,m+1)}(h_{k}f)(\Phi_{k}(\zeta))|\nu_{p}(\zeta)\d\zeta}\\
&\leq C_{0}\int_{U}{\sum_{k\in \mathcal{K}}|\partial^{(m+1,\ldots,m+1)}(h_{k}f)(\Phi_{k}(\zeta))|\nu_{p}(\zeta)\d\zeta}\\
&\underset{\mathclap{\eqref{thm13.0.5}}}{=}C_{0}\int_{U}{\bigl|\sum_{k\in \mathcal{K}}\partial^{(m+1,\ldots,m+1)}
(h_{k}f)(\Phi_{k}(\zeta))\bigr|\nu_{p}(\zeta)\d\zeta}
\end{flalign*}
with $C_{0}:=16^{dm}DA_{1}(I_{1}D_{3}I_{11})$. Therefore it follows that
\begin{align}\label{thm13.0.7}
\left|f\right|_{n,m}&=\sup_{\substack{x\in \Omega_{n}\\ \alpha\in \N^{d}_{0},\,|\alpha|\leq m}}
|\partial^{\alpha}f(x)|\nu_{n}(x)\nonumber\\
&\leq C_{0}\int_{U}{\bigl|\sum_{k\in \mathcal{K}}\partial^{(m+1,\ldots,m+1)}
(h_{k}f)(\Phi_{k}(\zeta))\bigr|\nu_{p}(\zeta)\d\zeta}.
\end{align}
Due to condition $(\omega.2)$, there are $I_{2}(p)\geq p$, $\psi_{p}\in L^{1}(\Omega_{n+1})$, $\psi_{p}>0$, and $A_{2}(p)>0$ 
such that $\nu_{p}(\zeta)\leq A_{2}(p)\psi_{p}(\zeta)\nu_{I_{2}(p)}(\zeta)$ for all $\zeta\in\Omega_{n+1}$. 
For $\zeta\in U\subset \Omega_{n+1}$ we set
\[
\mathfrak{I}(\zeta)[f]:=\sum_{k\in \mathcal{K}}\partial^{(m+1,\ldots,m+1)}(h_{k}f)(\Phi_{k}(\zeta))\nu_{I_{2}(p)}(\zeta).
\]
By \eqref{thm13.0.5}, there is at most one $k_{0}\in \mathcal{K}$ such that
$\zeta\in\operatorname{supp}([\partial^{(m+1,\ldots,m+1)}(h_{k_{0}}f)]\circ\Phi_{k_{0}})$.  
Let $\zeta\in U$ be such that there exists such a $k_{0}$ and set $\widetilde{m}:=(m+1,\ldots,m+1)$. 
Due to $(p.3)$ and the Leibniz rule, we have 
\begin{flalign*}
&\hspace{0.36cm}|\partial^{(m+1,\ldots,m+1)}(h_{k_{0}}f)(\Phi_{k_{0}}(\zeta))|\\
&=\bigl|\sum_{\gamma\leq\widetilde{m}}\dbinom{\widetilde{m}}{\gamma}\partial^{\widetilde{m}-\gamma}h_{k_{0}}(\Phi_{k_{0}}(\zeta))
\partial^{\gamma}f(\Phi_{k_{0}}(\zeta))\bigr|\\
&\leq \sum_{\gamma\leq\widetilde{m}}\dbinom{\widetilde{m}}{\gamma}|\partial^{\widetilde{m}-\gamma}h_{k_{0}}(\Phi_{k_{0}}(\zeta))|
\sup_{\substack{x\in B_{k_{0}}\\\beta\in\N^{d}_{0},\,|\beta|\leq |\widetilde{m}|}}|\partial^{\beta}f(x)|\\
&\underset{\mathclap{(p.3)}}{\leq}\;\; \sum_{\gamma\leq\widetilde{m}}\dbinom{\widetilde{m}}{\gamma}
\widetilde{C}_{\widetilde{m}-\gamma}\Bigl(\frac{1}{r_{n,3}(z_{k_{0}})}\Bigr)^{d+|\widetilde{m}-\gamma|}
\sup_{\substack{x\in B_{k_{0}}\\\beta\in\N^{d}_{0},\,|\beta|\leq d(m+1)}}|\partial^{\beta}f(x)|\\
&\leq \Bigl(\sum_{\gamma\leq\widetilde{m}}\dbinom{\widetilde{m}}{\gamma}\widetilde{C}_{\widetilde{m}-\gamma}\Bigr)
\Bigl(\frac{1}{r_{n,3}(z_{k_{0}})}\Bigr)^{d(m+2)}
\sup_{\substack{x\in B_{k_{0}}\\\beta\in\N^{d}_{0},\,|\beta|\leq d(m+1)}}|\partial^{\beta}f(x)|\\
&\leq C_{1}\Bigl(\frac{1}{r_{n,3}(z_{k_{0}})}\Bigr)^{d(m+2)}
\sup_{\substack{x\in B_{k_{0}}\\\beta\in\N^{d}_{0},\,|\beta|\leq d(m+1)}}|\partial^{\beta}f(x)|
\end{flalign*}
where the constant $C_{1}:=\sum_{\gamma\leq\widetilde{m}}\dbinom{\widetilde{m}}{\gamma}\widetilde{C}_{\widetilde{m}-\gamma}$ 
depends on $m$ but not on $\zeta$ or $k_{0}$. 
Hence we have with $q:=I_{1111}G_{3}I_{11112}(p)$, where $G_{3}$ is the $d(m+2)$-fold composition of $I_{3}$, 
\begin{align*}
|\mathfrak{I}(\zeta)[f]|&=\bigl|\sum_{k\in\mathcal{K}}\partial^{(m+1,\ldots,m+1)}
(h_{k}f)(\Phi_{k}(\zeta))\nu_{I_{2}(p)}(\zeta)|\\
&\leq |\partial^{(m+1,\ldots,m+1)}(h_{k_{0}}f)(\Phi_{k_{0}}(\zeta))|\nu_{I_{2}(p)}(\zeta)\\
&\leq C_{1}\sup_{\substack{x\in B_{k_{0}}\\\beta\in\N^{d}_{0},\,|\beta|\leq d(m+1)}}
|\partial^{\beta}f(x)|\frac{\nu_{I_{2}(p)}(\zeta)}{r_{n,3}(z_{k_{0}})^{d(m+2)}}\\
&\underset{\mathclap{\eqref{thm13.0.4},(p.4)}}{\leq}C_{1}D_{1}
\sup_{\substack{x\in B_{k_{0}}\\\beta\in\N^{d}_{0},\,|\beta|\leq d(m+1)}}
|\partial^{\beta}f(x)|\nu_{I_{111}G_{3}I_{11112}}(z_{k_{0}})\\
&\underset{\mathclap{\substack{(\omega.1)\\B_{k_{0}}\subset\Omega_{n+1}}}}{\leq}C_{1}D_{1}A_{1}|f|_{q+1,d(m+1)}
\end{align*}
for all $\zeta\in U$ such that there is $k_{0}\in\mathcal{K}$ with
$\zeta\in\operatorname{supp}([\partial^{(m+1,\ldots,m+1)}(h_{k_{0}}f)]\circ\Phi_{k_{0}})$. 
If there is no such $k_{0}\in \mathcal{K}$ for $\zeta\in U$, we have $\mathfrak{I}(\zeta)[f]=0$, 
which yields that the estimate 
\begin{equation}\label{thm13.0.8a}
|\mathfrak{I}(\zeta)[f]|\leq C_{1}D_{1}A_{1}|f|_{q+1,d(m+1)}
\end{equation}
holds for all $\zeta\in U$.\\
If we set $V:=\{f\in\mathcal{EV}(\Omega)\;|\; |f|_{q+1,d(m+1)}\leq \tfrac{1}{C_{1}D_{1}A_{1}}\}$, 
then $V$ is an absolutely convex neighbourhood of zero in $\mathcal{EV}(\Omega)$. 
We claim that the mapping
\[
\widetilde{\mathfrak{I}}:=\tfrac{1}{\nu_{I_{2}(p)}}\mathfrak{I}\colon U\to \mathcal{EV}(\Omega)_{\sigma}'
\]
is continuous where $\mathcal{EV}(\Omega)_{\sigma}'$ is the dual space of $\mathcal{EV}(\Omega)$ equipped 
with the weak$^{\ast}$-topology $\sigma^{\ast}$. 
Let $\varepsilon>0$, $\zeta\in U$ and $f\in\mathcal{EV}(\Omega)$. If there exists $k_{0}\in\mathcal{K}$ such that 
$\zeta\in\operatorname{supp}([\partial^{(m+1,\ldots,m+1)}(h_{k_{0}}f)]\circ\Phi_{k_{0}})\subset Q_{k_{0}}$ 
(there is at most one $k_{0}$ by \eqref{thm13.0.5}), 
then we choose $0<\delta_{1}<\d^{\infty}_{\partial Q_{k_{0}}}(\zeta)$. For all $x\in U$ such that 
$\|x-\zeta\|_{\infty}<\delta_{1}$ the following is valid:
\begin{flalign*}
&\hspace{0.36cm}|\widetilde{\mathfrak{I}}(x)[f]-\widetilde{\mathfrak{I}}(\zeta)[f]|\\
&=|\partial^{(m+1,\ldots,m+1)}(h_{k_{0}}f)(\Phi_{k_{0}}(x))-\partial^{(m+1,\ldots,m+1)}(h_{k_{0}}f)(\Phi_{k_{0}}(\zeta))|
\end{flalign*}
Since $\partial^{(m+1,\dots,m+1)}(h_{k_{0}}f)\circ\Phi_{k_{0}}$ is continuous on $U$, there is 
$\delta_{2}>0$ such that
\[
|\widetilde{\mathfrak{I}}(x)[f]-\widetilde{\mathfrak{I}}(\zeta)[f]|<\varepsilon
\]
for all $x\in U$ with $\|x-\zeta\|_{\infty}<\min(\delta_{1},\delta_{2})$. 
If there is no such $k_{0}\in\mathcal{K}$ for $\zeta\in U$, we proceed as follows. 
First, we show that the family $(\operatorname{supp}([\partial^{(m+1,\ldots,m+1)}(h_{k}f)]\circ\Phi_{k}))_{k\in\mathcal{K}}$
is locally finite in $U$. 
Let $x\in U$. Then there is $k\in\mathcal{K}$ such that $x\in B_{k}$. The set $B_{k}$ is a neighbourhood of $x$ in $U$ 
and by $(p.2)$ it intersects only finitely many $B_{j}$, $j\in\mathcal{K}$. We deduce from \eqref{thm13.0.3}, 
\eqref{thm13.0.4} and \eqref{thm13.0.5} that 
$B_{k}$ intersects only finitely many $\operatorname{supp}([\partial^{(m+1,\ldots,m+1)}(h_{j}f)]\circ\Phi_{j})$, 
$j\in\mathcal{K}$, as well, yielding the local finiteness.
Second, $S:=\bigcup_{k\in\mathcal{K}}\operatorname{supp}([\partial^{(m+1,\ldots,m+1)}(h_{k}f)]\circ\Phi_{k})$ 
is closed in $U$ as the union of a locally finite family of closed sets. 
Hence there is an open neighbourhood $O_{\zeta}\subset(U\setminus S)$ of $\zeta$ and 
$\widetilde{\mathfrak{I}}(\cdot)[f]=0$ on $O_{\zeta}$. Therefore $\widetilde{\mathfrak{I}}(\cdot)[f]$ 
is continuous in $\zeta\in U$ in both cases, which proves our claim. 
In particular, $\widetilde{\mathfrak{I}}$ is Borel measurable and so $\mathfrak{I}=\widetilde{\mathfrak{I}}\nu_{I_{2}(p)}$ is 
Lebesgue measurable since $\nu_{I_{2}(p)}$ is Lebesgue measurable and scalar multiplication is continuous.  
In addition, we have $\mathfrak{I}(U)\subset \{y\in\mathcal{EV}(\Omega)'\;|\;\forall\;f\in V:\;|y(f)|\leq 1\}=V^{\circ}$ 
by \eqref{thm13.0.8a} where $V^{\circ}$ is the polar set of $V$. Next, we set
\[
u\colon \mathcal{C}(V^{\circ},\R)\to\R,\; u(g):=C_{0}A_{2}(p)\int_{U}{g(\mathfrak{I}(\zeta))\psi_{p}(\zeta)\d\zeta},
\]
and claim that this map is a well-defined positive continuous linear functional where $\mathcal{C}(V^{\circ},\R)$ 
is the space of real-valued continuous functions on $(V^{\circ},\sigma^{\ast})$.
Due to the Alao\u{g}lu-Bourbaki theorem, $V^{\circ}$ is weak$^{\ast}$-compact, the open set $U$ is Lebesgue measurable 
and $g\circ\mathfrak{I}$ is Lebesgue measurable since $g$ is continuous and $\mathfrak{I}$ Lebesgue measurable.
Moreover, $u$ is positive and linear and, if we equip $\mathcal{C}(V^{\circ},\R)$ with the norm
\[
\|g\|:=\sup_{y\in V^{\circ}}|g(y)|,\quad g\in\mathcal{C}(V^{\circ},\R),
\]
continuous as
\begin{align*}
|u(g)|&\leq C_{0}A_{2}(p)\int_{U}{|g(\underbrace{\mathfrak{I}(\zeta)}_{\in V^{\circ}})|\psi_{p}(\zeta)\d\zeta}
\leq C_{0}A_{2}(p)\int_{U}{\sup_{y\in V^{\circ}}|g(y)|\psi_{p}(\zeta)\d\zeta}\\
&=C_{0}A_{2}(p)\int_{U}{\psi_{p}(\zeta)\d\zeta}\|g\|
\underset{(\omega.2)}{\leq} C_{0}A_{2}(p)\|\psi_{p}\|_{L^{1}(\Omega_{n+1})}\|g\|
\end{align*}
since $\psi_{p}\in L^{1}(\Omega_{n+1})$ and $U\subset\Omega_{n+1}$, proving our claim, 
where we denote by $\|\cdot\|_{L^{1}(\Omega_{n+1})}$ the norm on the Lebesgue space $L^{1}(\Omega_{n+1})$.
Thus there is a positive Radon measure $\mu$ on $(V^{\circ},\sigma^{\ast})$  
by \cite[7.6.1 Riesz Representation Theorem, p.\ 139]{Jarchow} such that
\begin{equation}\label{thm13.0.8}
u(g)=\int_{V^{\circ}}{g \,\d\mu}.
\end{equation}
Altogether, we obtain, keeping in mind that every $f\in\mathcal{EV}(\Omega)$ defines a continuous (linear) functional 
$J(f)\colon V^{\circ}\to\K$, $y\mapsto y(f)$, that
\begin{align*}
|f|_{n,m}&\underset{\mathclap{\eqref{thm13.0.7}}}{\leq} C_{0}A_{2}(p)\int_{U}{|\mathfrak{I}(\zeta)[f]|\psi_{p}(\zeta)\d\zeta}
=C_{0}A_{2}(p)\int_{U}{|J(f)(\mathfrak{I}(\zeta))|\psi_{p}(\zeta)\d\zeta}\\
&=u(|J(f)|)
\underset{\mathclap{\eqref{thm13.0.8}}}{=}\;\;\int_{V^{\circ}}{|J(f)|\d\mu}
=\int_{V^{\circ}}{|y(f)|\d\mu(y)}.
\end{align*}
Therefore $\mathcal{EV}(\Omega)$ is nuclear by \cite[4.1.5 Proposition, p.\ 71]{Pietsch1972}.
\end{proof}

Since nuclearity is inherited by (topological) subspaces of nuclear spaces due to \cite[Proposition 28.6, p.\ 347]{meisevogt1997}, 
all subspaces of $\mathcal{EV}(\Omega)$, e.g.\ spaces of holomorphic or harmonic functions, are nuclear as well if 
the family of continuous weights $\mathcal{V}$ satisfies $(\omega.1)$-$(\omega.3)$ and one of the conditions 
$(s.1)$-$(s.3)$ is fulfilled.

\begin{rem}
\prettyref{thm:nuclear} is still valid if we replace the condition of $\nu_{n}$, $n\in\N$, being continuous 
by being locally bounded away from zero, locally bounded and Lebesgue measurable 
due to \prettyref{rem:r_n_loc_bound_away} and the observation that $\mathfrak{J}$ is Lebesgue measurable 
if the functions $\nu_{n}$ are Lebesgue measurable. 
Concerning $(s.2)$, we note that we can switch to an equivalent system of seminorms on $\mathcal{EV}(\Omega)$ 
by replacing $\Omega_{n}$ by its closure $\overline{\Omega}_{n}$ if $\overline{\Omega}_{n}\subset\Omega_{n+1}$ for all $n\in\N$.
Further, we can replace the conditions $(s.1)$-$(s.3)$ by the condition that $r_{n,k}>0$ 
on $\Omega_{n}$ for all $n\in\N$ and all $k\in\{1,2,3\}$ 
as this is the implication from $(s.1)$-$(s.3)$ we need for \prettyref{lem:partition_unity}.
\end{rem}

\begin{rem}
\cite[Lemma, p.\ 144]{nakamura1968_2} and \cite[Theorem, p.\ 145]{nakamura1968_2}, treating the nuclearity of weighted spaces of holomorphic 
functions on an open set $\Omega\subset\C^{d}$, are incorrect. Namely, if $\Omega$ is bounded and the family of weights
$(M_{\alpha})_{\alpha\in A}$ only consists of one element which is continuously extendable on $\overline{\Omega}$, 
then the space $Z_{\Omega}\{M_{\alpha}\}$ is an infinite dimensional Banach space since it contains all 
polynomials. However, from \cite[Lemma, p.\ 144]{nakamura1968_2} and \cite[Theorem, p.\ 145]{nakamura1968_2} 
follows that $Z_{\Omega}\{M_{\alpha}\}$ is nuclear, which is a contradiction.
\end{rem}

\subsection*{Acknowledgements}
We thank M.\ Langenbruch for the hint at the conditions on the weights and the partition of 
unity in \cite{L4}. Further, we are thankful to the anonymous reviewer for the thorough review and helpful suggestions. 